\documentclass[11pt]{article}

\usepackage{amsfonts}
\usepackage{amsmath}
\usepackage{amssymb}
\usepackage{epsfig}
\usepackage{enumerate}
\usepackage{citesort}
\numberwithin{equation}{section} 

\title{High-order Nystr\"om discretizations for the solution of integral equation formulations of two-dimensional Helmholtz transmission problems}
\author{Yassine Boubendir, V\'{\i}ctor Dom\'{\i}nguez, 
Catalin Turc\\ \small 
New Jersey Institute of Technology, Universidad P\'ublica de Navarra, Spain,\\ \small
New Jersey Institute of Technology\\ \small
boubendi@njit.edu,\ victor.dominguez@unavarra.es,\ 
catalin.c.turc@njit.edu}

\newtheorem{theorem}{Theorem}[section]
\newtheorem{lemma}[theorem]{Lemma}

\newtheorem{remark}[theorem]{Remark}
\newenvironment{proof}{\hspace{0.5cm} {\bf Proof.}}
{$\quad {}_\blacksquare$\vspace{0.3cm}}

\date{}
\newcommand{\triple}[1]{{\left\vert\kern-0.25ex\left\vert\kern-0.25ex\left\vert #1 
    \right\vert\kern-0.25ex\right\vert\kern-0.25ex\right\vert}}

\begin{document}
\maketitle
\begin{abstract}
  We present and analyze fully discrete Nystr\"om methods for the solution of three classes of well conditioned boundary integral equations for the solution of
 two dimensional scattering problems by homogeneous dielectric scatterers. Specifically, we perform the stability analysis of Nystr\"om discretizations of (1) the classical second kind integral equations for transmission problems~\cite{KressRoach}, (2) the {\em single} integral equation formulations~\cite{KleinmanMartin}, and (3) recently introduced Generalized Combined Source Integral Equations~\cite{turc2}. The Nystr\"om method that we use for the discretization of the various integral equations under consideration are based on global trigonometric approximations, splitting of the kernels of integral operators into singular and smooth components, and explicit quadratures of products of singular parts (logarithms) and trigonometric polynomials. The discretization of the integral equations (2) and (3) above requires special care as these formulations feature compositions of boundary integral operators that are pseudodifferential operators of positive and negative orders respectively. We deal with these compositions through Calder\'on's calculus and we establish the convergence of fully discrete Nystr\"om methods in appropriate Sobolev spaces which implies pointwise convergence of the discrete solutions. In the case of analytic boundaries, we establish superalgebraic convergence of the method.  \newline \indent
   
\end{abstract}

\section{Introduction\label{intro}}

\parskip 2pt plus2pt minus1pt

Numerical methods based on integral equation formulations for the solution of 
Helmholtz transmission problems, when applicable, have certain advantages over those that 
use volumetric formulations, largely owing to  the dimensional reduction, the 
explicit enforcement of the radiation conditions, and the absence of dispersion errors. Constructing integral equation formulations that are equivalent with the system of PDEs that models transmission scattering problems is by now well understood. Indeed, a wide variety of well-posed boundary integral equations for the solution of Helmholtz transmission problems has been proposed in the literature, at least in the case when the interfaces of material discontinuity are regular enough. Most of these formulations are derived from representations of the fields in each region filled by a homogeneous material by suitable combinations of single and double layer potentials. The enforcement of the continuity of the fields and their normal derivatives across interfaces of material discontinuity leads to Combined Field Integral Equations (CFIE) of transmission scattering problems. Some of these integral formulations involve two unknowns per each interface of material discontinuity~\cite{costabel-stephan,KittapaKleinman,KressRoach,LaRaSa:2009,rokhlin-dielectric}, while others involve one unknown per each interface of material discontinuity~\cite{KleinmanMartin}.  

Motivated by the quest to design integral equation formulations that have better spectral properties than those of the classical CFIE formulations, a new methodology that uses representations of fields in terms of suitable combinations of single and double layer potentials that act on certain regularizing operators has been proposed in the literature in the past ten years~\cite{AntoineX,Antoine,br-turc,turc1,Levadoux,turc1,turc2}. Typically the regularizing operators can be constructed using coercive approximations of Dirichlet-to-Neumann operators, see~\cite{turc3} for a in-depth discussion on this methodology for the case of Helmholtz transmission problems. The ensuing integral equations which are referred to as Generalized Combined Source Integral Equations (GCSIE) or Regularized Combined Field Integral Equations (CFIER) lead upon discretization to solvers that deliver important computational gains over solvers based on the classical CFIE, see for instance Section~\ref{prec}. While the design of GCSIE is quite well understood and can be carried out in a well defined program~\cite{turc3}, the stability/error analysis of numerical schemes based on the GCSIE formulations has not been pursued to a great extent in the literature. One difficulty that arises in the error analysis of numerical schemes based on GCSIE formulations is related to the fact that the latter formulations feature compositions of boundary integral operators that are pseudodifferential operators of positive and negative orders, which must be handled with care in order to lead to stable discretizations.

We present a Nystr\"om method for the discretization of our two dimensional GCSIE operators that follows the discretization method introduced in~\cite{KressH}. Under the assumption that the interface of material discontinuity is a regular enough closed curve, this algorithm is based on global trigonometric approximations, splitting of the kernels of integral operators into singular and smooth components, and explicit quadratures of products of singular parts (logarithms) and trigonometric polynomials~\cite{kusmaul,martensen}. Other numerical approximations which can be found in the scientific
literature consider Petrov-Galerkin schemes with piecewise polynomial
functions, like periodic splines which, in some way, can include trigonometric 
polynomial as a limit case
as the degree of the splines tends to infinity  cf. \cite{LaRaSa:2009,RaSa:2006a}.  In the same
frame, Dirac
deltas can be understood as splines of degree $-1$, which gives rise to quadrature 
methods cf. \cite{DoRaSa:2008}.

The main ingredients in the error analysis proof are the mapping properties of the boundary integral operators that enter the boundary integral formulations of Helmholtz transmission problems and Sobolev spaces bounds of the error in trigonometric interpolation. Helmholtz transmission integral equations require the use of all four boundary layer operators associated to the Helmholtz equation. We show how these operators can be fully discretized and used in all the formulations considered in this paper, some of which include compositions of some of these operators. We note that the discretization operator compositions can be handled with ease by collocation discretizations, as it simply amounts to matrix multiplications. The same objective were pursed in the recent papers 
\cite{DoLuSa:2014a,DoLuSa:2014b} where simpler yet still moderate order (2 and 3 respectively) discretizations of the 
layer operators were presented. These discretizations rely on geometric quantities only, and do not make use of any splitting of the kernels of the four boundary integral operators related to the Helmholtz equation.

As a consequence, we establish the convergence of the fully discrete GCSIE method in appropriate Sobolev spaces which implies pointwise convergence of the discrete solutions. In the case of analytic boundaries, we establish superalgebraic convergence of the Nystr\"om method. The same techniques outlined above allow us to carry the stability analysis of the Nystr\"om discretization of the {\em single} integral equations introduced in~\cite{KleinmanMartin}, which to the best of our knowledge did not exist in the literature thus far.  Given that solvers based on GCSIE and {\em single} formulations can lead to one order of magnitude faster numerics than those based on classical CFIE formulations~\cite{costabel-stephan,KittapaKleinman,KressRoach,rokhlin-dielectric} (see the numerical results in Section~\ref{prec}), the comprehensive error analysis we undergo in this paper can only strengthen the claim that the former formulations should be the formulations of choice when solving transmission scattering problems.

The paper is organized as follows: in Section~\ref{cfie} we review four boundary integral equation formulations for the solution of transmission scattering problems; in Section~\ref{singular_int} we present a 
Nystr\"om discretization of the boundary integral equations considered in Section~\ref{cfie} and we establish the high-order of convergence of our solvers; in Section~\ref{prec} we present a comparison of the properties of Nystr\"om integral solvers based on the various formulations discussed in this paper.  

\section{Integral Equations of Helmholtz transmission problems\label{cfie}}

We consider the problem of evaluating the time-harmonic fields $u^1$ and $u^2$ that result as an incident field $u^{inc}$ impinges upon the boundary
$\Gamma$ of a homogeneous dielectric scatterer $D_2$ which occupies a bounded region in $\mathbb{R}^2$. We assume that both media occupying $D_2$ and its exterior are nonmagnetic, and the electric permitivity of the dielectric material inside the domain $D_2$ is denoted by $\epsilon_2$ while that of the medium occupying the exterior of $D_2$ is denoted by $\epsilon_1$. The frequency domain dielectric transmission problem is formulated in terms of finding fields $u^1$ and $u^2$ that are solutions to the Helmholtz equations
\begin{equation} 
  \label{eq:Ac_i}
\begin{aligned}
  \Delta u^2+k_2^2 u^2&=& 0, \qquad &\mathrm{in}\ D_2,\\
  \Delta u^1+k_1^2 u^1&=&0,\qquad &\mathrm{in}\ D_1=\mathbb{R}^2\setminus {\overline{D_2}},
\end{aligned}
\end{equation}
given an incident field $u^{inc}$ that satisfies
\begin{equation}
  \label{eq:Maxwell_inc}
  \Delta u^{inc}+k_1^2 u^{inc}=0 \qquad \mathrm{in}\ D_1,
\end{equation}
where the wavenumbers $k_i,i=1,2$ are defined as $k_i=\omega\sqrt{\epsilon_i}, i=1,2$ in terms of the frequency $\omega$. In addition, the fields $u^{1}$, $u^{inc}$, and $u^2$ are related on the boundary $\Gamma$ by the the following boundary conditions
\begin{eqnarray}
\label{eq:bc}
\gamma_D^1 u^1 + \gamma_D^1 u^{inc} &=&\gamma_D^2 u^2\qquad \rm{on}\ \Gamma \nonumber\\
\gamma_N^1 u^1 + \gamma_N^1 u^{inc}&=&\nu \gamma_N^2u^2\qquad \rm{on}\ \Gamma.
\end{eqnarray}
In equations~\eqref{eq:bc} and what follows $\gamma_D^i,i=1,2$ denote exterior and respectively interior Dirichlet traces, whereas $\gamma_N^i,i=1,2$ denote exterior and respectively interior Neumann traces taken with respect to the exterior unit normal on $\Gamma$. We assume in what follows that the boundary $\Gamma$ is a closed and smooth curve in $\mathbb{R}^2$. Depending on the type of scattering problem, the transmission coefficient $\nu$ in equations~\eqref{eq:bc} can be either $1$ (E-polarized) or $\epsilon_1/\epsilon_2$ (H-polarized). We furthermore require that $u^1$ satisfies Sommerfeld radiation conditions at infinity:
\begin{equation}\label{eq:radiation}
\lim_{|r|\to\infty}r^{1/2}(\partial u^1/\partial r - ik_1u^1)=0.
\end{equation}
We assume in what follows that the wavenumbers $k_i,i=1,2$ are real. Under this assumption, it is well known that the systems of partial differential equations~\eqref{eq:Ac_i}-\eqref{eq:Maxwell_inc} together with the boundary conditions~\eqref{eq:bc} and the radiation condition~\eqref{eq:radiation} has a unique solution~\cite{KressRoach,KleinmanMartin}. Moreover, the adjoint problem obtained by interchanging the interior and exterior domains has a unique solution~\cite{KressRoach}. The results in this text can be extended to the case of complex wavenumbers $k_i,i=1,2$, provided we assume uniqueness of the transmission problem and its adjoint.

\subsection{Boundary integral operators associated with the Helmholtz equations and second kind boundary integral formulations of Helmholtz transmission problems\label{di_ind_cfie}}
A variety of well-posed integral equations for the transmission problem~\eqref{eq:Ac_i}-\eqref{eq:bc} exist~\cite{KressRoach,costabel-stephan,KleinmanMartin,turc2}. On one hand, integral equations formulations for transmission problems can be formulated as a $2\times 2$ system of integral equations which can be derived from either (a) Green's formulas in both domains $D_1$ and $D_2$, in which case they are referred to as {\em direct} integral equation formulations~\cite{costabel-stephan,KleinmanMartin}, (b) from representations of the fields $u^j,j=1,2$ in forms of suitable combinations of single and double layer potentials in both domains $D_1$ and $D_2$, in which case they are referred to as {\em indirect} integral equation formulations~\cite{KressRoach}, (c) from Green's formulas and suitable approximations to exterior and interior Dirichlet-to-Neumann operators, in which case they are referred to as {\em regularized} combined field integral equations or generalized combined source integral equations~\cite{turc2}. On the other hand, integral equations formulations for transmission problems can be formulated as {\em single} integral equations which can be derived from (d) Green's formulas in one of the domains and (indirect) combined field representations in the other domain~\cite{KleinmanMartin}. The strategies recounted above lead to Fredholm second kind boundary integral equations for the solution of transmission problems~\cite{KressRoach,KleinmanMartin,turc2}, at least in the case when the curve $\Gamma$ is smooth enough ($C^3$ suffices). In order to present the aforementioned integral formulations, we review first the definition and mapping properties of the various scattering boundary integral operators. 

We start with the definition of the single and double layer potentials. Given a wavenumber $k$ such that $\Re{k}>0$ and $\Im{k}\geq 0$, and a density $\varphi$ defined on $\Gamma$, we define the single layer potential as
$$[SL_k(\varphi)](\mathbf{z}):=\int_\Gamma G_k(\mathbf{z}-\mathbf{y})\varphi(\mathbf{y})ds(\mathbf{y}),\ \mathbf{z}\in\mathbb{R}^2\setminus\Gamma$$
and the double layer potential as
$$[DL_k(\varphi)](\mathbf{z}):=\int_\Gamma \frac{\partial G_k(\mathbf{z}-\mathbf{y})}{\partial\mathbf{n}(\mathbf{y})}\varphi(\mathbf{y})ds(\mathbf{y}),\ \mathbf{z}\in\mathbb{R}^2\setminus\Gamma$$
where $G_k(\mathbf{x})=\frac{i}{4}H_0^{(1)}(k|\mathbf{x}|)$ represents the two-dimensional Green's function of the Helmholtz equation with wavenumber $k$. The Dirichlet and Neumann exterior and interior traces on $\Gamma$ of the single and double layer potentials corresponding to the wavenumber $k$ and a density $\varphi$ are given by
\begin{eqnarray}\label{traces}
\gamma_D^1 SL_k(\varphi)&=&\gamma_D^2 SL_k(\varphi)=S_k\varphi \nonumber\\
\gamma_N^j SL_k(\varphi)&=&(-1)^j\frac{\varphi}{2}+K_k^\top \varphi\quad j=1,2\nonumber\\
\gamma_D^j DL_k(\varphi)&=&(-1)^{j+1}\frac{\varphi}{2}+K_k\varphi\quad j=1,2\nonumber\\
\gamma_N^1 DL_k(\varphi)&=&\gamma_N^2 DL_k(\varphi)=N_k\varphi.
\end{eqnarray}
In equations~\eqref{traces} the operators $K_k$ and $K^\top_k$, usually
referred to as double and adjoint double layer operators, are defined for a given wavenumber $k$ and density $\varphi$ as
\begin{equation}
\label{eq:double}
(K_k\varphi)(\mathbf x):=\int_{\Gamma}\frac{\partial G_k(\mathbf x-\mathbf y)}{\partial\mathbf{n}(\mathbf y)}\varphi(\mathbf y)ds(\mathbf y),\ \mathbf x\ {\rm on}\ \Gamma
\end{equation}
and 
\begin{equation}
\label{eq:adj_double}
(K_k^\top\varphi)(\mathbf x):=\int_{\Gamma}\frac{\partial G_k(\mathbf x-\mathbf y)}{\partial\mathbf{n}(\mathbf x)}\varphi(\mathbf y)ds(\mathbf y),\ \mathbf x\ {\rm on}\ \Gamma.
\end{equation}
Furthermore, for a given wavenumber $k$ and density $\varphi$, 
the operator $N_k$ denotes the Neumann trace of the double layer potential on 
$\Gamma$ given in terms of a Hadamard Finite Part (FP)  integral which can be re-expressed in terms of a Cauchy Principal Value (PV) integral that involves the tangential derivative $\partial_s$ on the curve $\Gamma$
\begin{eqnarray}
 \label{eq:normal_double}
(N_k \varphi)(\mathbf x) &:=& \text{FP} \int_\Gamma \frac{\partial^{2}G_k(\mathbf x -\mathbf y)}{\partial \mathbf{n}(\mathbf x) \partial \mathbf{n}(\mathbf y)} \varphi(\mathbf y)ds(\mathbf y) 
\nonumber\\
&=&k^{2}\int_\Gamma G_k(\mathbf x -\mathbf y)
(\mathbf{n}(\mathbf x)\cdot\mathbf{n}(\mathbf y))\varphi(\mathbf y)ds(\mathbf y)+ {\rm PV}
\int_\Gamma \partial_s G_k(\mathbf x -\mathbf y)\partial_s \varphi(\mathbf y)ds(\mathbf y).\nonumber 
\end{eqnarray}
Finally, the single layer operator $S_k$ is defined for a wavenumber $k$ as
\begin{equation}\label{eq:sl}
(S_k\varphi)(\mathbf x):=\int_\Gamma G_k(\mathbf x -\mathbf y)\varphi(\mathbf y)ds(\mathbf y),\ \mathbf{x}\ {\rm on} \ \Gamma
\end{equation} 
for a density function $\varphi$ defined on $\Gamma$. Having recalled the definition of the scattering boundary integral operators, we present next their mapping properties in appropriate Sobolev spaces of functions defined on the curve $\Gamma$.

In what follows we recall the definition of Sobolev spaces $H^p(\Gamma)$ according 
to~\cite{Kress,Saranen}, as we make frequent use of these spaces. Our presentation 
follows very closely that in~\cite[Ch. 8]{Kress}. We then define, for $p\ge 0$, the $2\pi$ periodic Sobolev space of order $p$ 
\[
 H^p[0,2\pi]:=\{\varphi\in L^2[0,2\pi]\ |\ \|\varphi\|_p<\infty\}
\]
where
\begin{equation}
\label{eq:sobolev}
\| \varphi\|_p^2:=\sum_{m=-\infty}^\infty 
(1+|m|^{2})^p|\widehat{\varphi}(m)|^2,\quad \text{with}\quad 
 \widehat{\varphi}(m):=\frac{1}{2\pi}\int_0^{2\pi}\varphi(t) e^{-im t}\,dt. 
\end{equation}
is the (periodic) Sobolev norm. 
Clearly, $H^p[0,2\pi]$ equipped with  the natural inner product is a Hilbert space, with $H^0[0,2\pi]=L^2[0,2\pi]
$. For $p<0$, the same construction can be easily adapted to define $H^p[0,2\pi]$, once the Fourier coefficients
are understood in a weak sense. Equivalently, one can introduce  $H^{p}[0,2\pi]$, for $p<0$ 
as the dual space of $H^{-p}[0,2\pi]$, that is the space of bounded linear functionals on $H^{-p}[0,2\pi]$. If the curve $\Gamma$ is represented by a smooth (infinitely differentiable) 
$2\pi$ periodic parametrization $\Gamma=\{{\bf x}(t): t\in[0,2\pi)\}$, then 
the space $H^p(\Gamma)$ is defined as the space of functions 
$\varphi\in L^2(\Gamma)$ such that $\varphi\circ {\bf x}\in H^p[0,2\pi]$. 
It is a classical result, see for instance~\cite{Kress} that this definition of 
$H^p(\Gamma)$ is invariant with respect to the parametrization.

Having reviewed the definition of Sobolev spaces $H^{s}(\Gamma)$, 
we recall in the next result the mapping properties of the boundary integral operators defined above~\cite{turc2}:
\begin{theorem}\label{regL} 
For smooth and closed curves $\Gamma$ and all $s\in\mathbb{R}$ it holds
\begin{itemize}
\item $S_k:H^{s}(\Gamma)\to H^{s+1}(\Gamma)$
\item $N_k:H^{s}(\Gamma)\to H^{s-1}(\Gamma)$
\item $K_k^\top:H^{s}(\Gamma)\to H^{s+3}(\Gamma)$
\item $K_k:H^{s}(\Gamma)\to H^{s+3}(\Gamma)$.
\end{itemize}
In addition, for $\kappa\ne \kappa_2$, the operator $S_{\kappa}-S_{\kappa_2}$ is 
regularizing of three orders, that is 
$S_{\kappa}-S_{\kappa_2}:H^{s}(\Gamma)\to H^{s+3}(\Gamma)$ and 
$N_{\kappa}-N_{\kappa_2}$ is regularizing of one order, that is 
$N_{\kappa}-N_{\kappa_2}:H^{s}(\Gamma)\to H^{s+1}(\Gamma)$.
\end{theorem}

A widely used boundary integral formulation of the transmission problem~\eqref{eq:Ac_i}-\eqref{eq:bc} consists of the the following pair of integral equations whose two unknowns are values of the total exterior field $u = u^1+u^{inc}$ and
its normal derivative $\frac{\partial u}{\partial n}$ on $\Gamma$:
\begin{equation}\label{eq:system_trans}
\begin{array}{rcl}
 \displaystyle
  \frac{\nu^{-1}+1}{2}u(\mathbf x) + (K_2-\nu^{-1}K_1)(u)(\mathbf x) +\nu^{-1} (S_1 - S_2)\left(\frac{\partial u}{\partial n}\right)(\mathbf x)&=& \displaystyle\nu^{-1} u^{inc}(\mathbf x) \\
   \displaystyle
  \frac{\nu^{-1}+1}{2}\frac{\partial u}{\partial n}(\mathbf x) + (K_1^\top-\nu^{-1} K^\top_2)\left(\frac{\partial u}{\partial n}\right)(\mathbf x) -(N_1 - N_2)(u)(\mathbf x)&=& \displaystyle\frac{\partial u^{inc}}{\partial n}(\mathbf x),
\end{array}
\end{equation}
($\mathbf x\in\Gamma$). In view of the results in Theorem~\ref{regL}, this system is Fredholm of the second kind in $H^s(\Gamma)\times H^s(\Gamma)$. In addition, the system~\eqref{eq:system_trans} can be shown to be uniquely solvable~\cite{KleinmanMartin}. In what follows we refer to the integral equations~\eqref{eq:system_trans} by CFIESK. We introduced recently regularized combined field integral equation formulations of transmission problems~\cite{turc2}. We look for fields $(u^1,u^2)$ defined as
\begin{eqnarray*}
u^1(\mathbf{z})&=& 
DL_1(\widetilde{R}_{11}a+\widetilde{R}_{12}b)(\mathbf{z})-SL_1(\widetilde{R}_{21
}a+\widetilde{R}_{22} b)(\mathbf{z}),\quad 
\mathbf{z}\in D_1\\
u^2(\mathbf{z})&=& -DL_2(\widetilde{R}_{11} a+\widetilde{R}_{12} 
b-a)(\mathbf{z})+\nu^{-1}SL_2(\widetilde{R}_{21} a+\widetilde{R}_{22} 
b-b)(\mathbf{z}),\quad \mathbf{z}\in D_2,
\end{eqnarray*} 
where the regularizing operators $\widetilde{R}_{ij},1\leq i,j\leq 2$ are defined as
\begin{equation}
\label{eq:tildR1}
\begin{array}{rclrcl} 
\widetilde{R}_{11}&:=&\displaystyle\frac{\nu}{1+\nu}I,\quad&
\widetilde{R}_{12}&:=&\displaystyle-\frac{2}{1+\nu}S_{\kappa}\\[1.25ex]
\widetilde{R}_{21}&:=&\displaystyle\frac{2\nu}{1+\nu}N_{\kappa},\quad&
\widetilde{R}_{22}&:=&\displaystyle\frac{1}{1+\nu}I,
\end{array}
\end{equation}
where $\kappa=\kappa_1 + i\varepsilon$ with $\kappa_1>0$ and $\varepsilon>0$. The enforcement of transmission boundary conditions~\eqref{eq:bc} leads to the following Generalized Combined Source Integral Equations (GCSIE) that are uniquely solvable in $H^{s}(\Gamma)\times 
H^{s}(\Gamma)$~\cite{turc2}: 
\begin{eqnarray}\label{eq:matrix_explicit}
\left(\begin{array}{cc}\tilde{D}_{11}&\tilde{D}_{12}\\\tilde{D}_{21}&\tilde{D}_{22}\end{array}\right)\left(\begin{array}{c}a\\b\end{array}\right)&=&-\left(\begin{array}{c}\gamma_D^1 u^{inc}\\\gamma_N^1 u^{inc}\end{array}\right)\nonumber\\
\tilde{D}_{11}&:=&I-\frac{1}{1+\nu}K_2+\frac{\nu}{1+\nu}K_1-\frac{2\nu}{1+\nu}S_1(N_{\kappa}-N_1)-\frac{2\nu}{1+\nu}(K_1)^2\nonumber\\
&&-\frac{2}{1+\nu}S_2(N_{\kappa}-N_2)-\frac{2}{1+\nu}(K_2)^2\nonumber\\
\tilde{D}_{12}&:=&\frac{1}{1+\nu}(S_2-S_1)-\frac{2}{1+\nu}(K_1+K_2)S_{\kappa}\nonumber\\
\tilde{D}_{21}&:=&\frac{\nu}{1+\nu}(N_1-N_2)-\frac{2\nu}{1+\nu}(K_1^T+K_2^T)N_{\kappa}\nonumber\\
\tilde{D}_{22}&:=&I+\frac{\nu}{1+\nu}K_2^T-\frac{1}{1+\nu}K_1^T-\frac{2}{1+\nu}(N_1-N_{\kappa})S_{\kappa}\nonumber\\
&&-\frac{2\nu}{1+\nu}(N_2-N_{\kappa})S_{\kappa}-2(K_{\kappa}^T)^2.
\end{eqnarray}
If we replace the regularizing operators in equations~\eqref{eq:tildR1} by their {\em periodic} principal symbols in the sense of pseudodifferential operators~\cite{Saranen,Taylor}, we obtain regularized formulations that are uniquely solvable in the spaces $H^s[0,2\pi]\times H^s[0,2\pi]$~\cite{turc2}. Define
\begin{equation}\label{eq:PS}
\sigma_0(N_{\kappa})(\xi)=-\frac{1}{2}\sqrt{|\xi|^2-\kappa^2}\qquad \sigma_0(S_{\kappa})(\xi)=\frac{1}{2\sqrt{|\xi|^2-\kappa^2}},
\end{equation}
where $\kappa=\kappa_1+i\varepsilon$ with $\kappa_1>0$ and $\varepsilon>0$ and where the square roots are chose in equation~\eqref{eq:PS} so that $\Im(\sigma_0(N_{\kappa})) >0$ and $\Im(\sigma_0(S_{\kappa}))>0$. We define then the operators 
\begin{equation}\label{eq:defPS1}
(PS(N_{\kappa})\phi)({\bf x}(t)):=\frac{1}{|{\bf x}'(t)|}\sum_{n\in\mathbb{Z}}\sigma_0(N_{\kappa})(n)\widehat{\phi}_n e^{int},\quad \widehat{\phi}_n:=\int_0^{2\pi} (\phi\circ{\bf x})(\tau) e^{-in\tau}\,{\rm d}\tau 
\end{equation} 
and
\begin{equation}\label{eq:defPS2}
(PS(S_{\kappa})\psi)({\bf x}(t)):=\sum_{n\in\mathbb{Z}}\sigma_0(S_{\kappa})(n) \tilde{\psi}_n
e^{int} 
\quad \tilde{\psi}_n:=\int_0^{2\pi} (\psi\circ {\bf x})(\tau))|\mathbf x'(\tau)| e^{-in\tau}\,{\rm d}\tau  
\end{equation}
for $2\pi-$periodic functions.  It follows from their definition that 
$PS(N_{\kappa}):H^s(\Gamma)\to H^{s-1}(\Gamma)$ and 
$PS(S_{\kappa}):H^s(\Gamma)\to H^{s+1}(\Gamma)$.  


Now we seek for $(a^1, b^1)$ so that the solution of 
\eqref{eq:Ac_i}--\eqref{eq:bc} can be written as
\begin{eqnarray*}
u^1&=& 
DL_1[PS(\widetilde{R}_{11})a^1+PS(\widetilde{R}_{12})b^1]-SL_1[PS(\widetilde{R}_{21
})a^1+PS(\widetilde{R}_{22}) b^1],\quad 
\text{in $D_1$}\\
u^2 &=& -DL_2[PS(\widetilde{R}_{11}) a^1+PS(\widetilde{R}_{12}) 
b^1-a^1] +\nu^{-1}SL_2[PS(\widetilde{R}_{21}) a^1+PS(\widetilde{R}_{22}) 
b^1-b^1],\quad 
\text{in $D_2$}
\end{eqnarray*} 
where $PS(\widetilde{R}_{ij}),1\leq i,j\leq 2$ are regularizing operators  defined as 
\[
\begin{array}{rclrcl} 
PS(\widetilde{R}_{11})&:=&\displaystyle\frac{\nu}{1+\nu}I,\quad&
PS(\widetilde{R}_{12})&:=&\displaystyle-\frac{2}{1+\nu}PS(S_{\kappa})\\[1.25ex]
PS(\widetilde{R}_{21})&:=&\displaystyle\frac{2\nu}{1+\nu}PS(N_{\kappa}),\quad&
PS(\widetilde{R}_{22})&:=&\displaystyle\frac{1}{1+\nu}I.
\end{array} 
\]
The enforcement of transmission boundary conditions~\eqref{eq:bc} leads to the following Principal Symbol Generalized Combined Source Integral Equations (PSGCSIE) that are uniquely solvable in $H^{s}[0,2\pi]\times 
H^{s}[0,2\pi]$~\cite{turc2}: 
\begin{eqnarray}\label{eq:matrix_PSexplicit}
\left(\begin{array}{cc}PS\tilde{D}_{11}&PS\tilde{D}_{12}\\PS\tilde{D}_{21}&PS\tilde{D}_{22}\end{array}\right)\left(\begin{array}{c}a^1\\b^1\end{array}\right)&=&-\left(\begin{array}{c}\gamma_D^1 u^{inc}\\\gamma_N^1 u^{inc}\end{array}\right)
\end{eqnarray}
where
\begin{eqnarray}\label{eq:entriesPSexplicit}
PS\tilde{D}_{11}&:=&\frac{1}{2}I+\frac{\nu}{1+\nu}K_1-\frac{1}{1+\nu}K_2-\frac{2\nu}{1+\nu}(S_1+\nu^{-1}S_2)PS(N_{\kappa})\nonumber\\
PS\tilde{D}_{12}&:=&\frac{1}{1+\nu}(S_2 -S_1)-\frac{2}{1+\nu}(K_1+K_2)PS(S_{\kappa})\nonumber\\
PS\tilde{D}_{21}&:=&\frac{\nu}{1+\nu}(N_1-N_2)-\frac{2\nu}{1+\nu}(K^T_1+K^T_2)PS(N_{\kappa})\nonumber\\
PS\tilde{D}_{22}&:=&\frac{1}{2}I+\frac{\nu}{1+\nu}K^T_2-\frac{1}{1+\nu}K_1^T-\frac{2}{1+\nu}(N_1+\nu N_2)PS(S_{\kappa}).
\end{eqnarray}

Another possible formulations of the transmission problem~\eqref{eq:Ac_i}-\eqref{eq:bc} take on the form of {\em single} integral equations~\cite{KleinmanMartin}. Amongst several possible choices of such equations, we consider the following version (equation (7.4) in~\cite{KleinmanMartin}) which we have observed in practice to lead to better spectral properties. The main idea is to look for the field $u^2$ as a single layer potential, that is
\[
u^2(\mathbf{z})=-2[SL_2\varphi](\mathbf{z}),\ \mathbf{z}\in D_2
\] 
and use the transmission boundary conditions~\eqref{eq:bc} and the Green's identities to express $u^1$ in the form
\[
u^1(\mathbf{z})=\nu SL_1[(I+2K_2^\top)\varphi](\mathbf{z})-2DL_1[S_2\varphi](\mathbf{z}),\ \mathbf{z}\in D_1.
\]
It follows that
\[
\gamma_D^2u^2 = -2S_2\varphi\qquad \gamma_N^2u^2=-(I+2K_2^\top)\varphi
\]
and
\[
\gamma_D^1u^1=\nu S_1(I+2K_2^\top)\varphi-S_2\varphi-2K_1S_2\varphi\qquad 
\gamma_N^1u^1=-{\frac{\nu}{2}}\varphi-\nu K_2^\top\varphi + \nu K_1^\top(I+2K_2^\top)\varphi-2N_1S_2\varphi.
\]
Using these representations of the fields $u^1$ and $u^2$, the Neumann and Dirichlet traces of $u^1$ and $u^2$ on $\Gamma$ are used in a Burton-Miller type combination of the form
$-\left(\gamma_N^1u^1-i\eta \gamma_D^1u^1\right)+\left(\nu\gamma_N^2u^2-i\eta \gamma_D^2u^2\right)=\gamma_N^1u^{inc}-i\eta \gamma_D^1u^{inc}$~\cite[Eq.(7.4)]{BurtonMiller} to lead to the following boundary integral equation  
\begin{equation}\label{eq:single}
-\frac{1+\nu}{2}\varphi+\mathbf{K}\varphi-i\eta\mathbf{S}\varphi=\frac{\partial u^{inc}}{\partial n}-i\eta u^{inc},\quad \eta\in\mathbb{R}\,\quad \eta\neq 0,
\end{equation}
where  must use Calder\'on's identity $N_2S_2=-\frac14I + (K_2^\top)^2$
$$\mathbf{K}=-K_2^\top(\nu I-2K_2^\top)-\nu K_1^\top(I+2K_2^\top)+2(N_1-N_2)S_2$$
and
$$\mathbf{S}=-\nu S_1(I+2K_2^\top)-(I-2K_1)S_2,$$
We refer in what follows to equation~\eqref{eq:single} by SCFIE. The coupling parameter $\eta$ in equations~\eqref{eq:single} is typically taken to be equal to $k_1$. We present in next section a Nystr\"om method of discretization of all of the formulations CFIESK, GCSIE, PSGCSIE, and SCFIE.

\section{Numerical method\label{singular_int}}

\parskip 10pt plus2pt minus1pt
\parindent0pt

We present in this section Nystr\"om discretizations of GCSIE formulations~\eqref{eq:matrix_explicit} and PSGCSIE formulations~\eqref{eq:matrix_PSexplicit}. These discretizations can then be applied for the other two formulations CFIESK~\eqref{eq:system_trans} and SCFIE~\eqref{eq:single}. The Nystr\"om discretizations are based on extensions of the Nystr\"om discretization introduced in~\cite{KressH} that were recently used in~\cite{turc1}. We also derive error estimates for the solutions that are obtained through these discretizations.

\subsection{Parametrized integral layer operators }

Recall we have assumed that the boundary curve $\Gamma$ is smooth and we have a smooth $2\pi$ periodic parametrization $\mathbf{x}(t)=(x_1(t),x_2(t))$. That is,  $x_j:\mathbb{R}\to\mathbb{R}$ are analytic and $2\pi$ 
periodic with $|{\bf x}'(t)|>0$ for all $t$.


We can then introduce the parameterized version of the integral layer operators. Hence, we have first the parametrized single layer operator \eqref{eq:sl}
\begin{equation}\label{eq:sl:b}
(S_k\varphi)(t)=\int_0^{2\pi} M_k(t,\tau)\varphi(\tau)d\tau:=
\int_0^{2\pi} G_k(\mathbf x(t) -\mathbf x(\tau))\varphi(\tau)d\tau,
\end{equation} 
where $\varphi$ it is a sufficiently smooth $2\pi-$periodic function. Note that 
the norm of parametrization $|{\bf x}'(t)|$ does not appear in the kernel 
in \eqref{eq:sl:b}. Thus, we are assuming that it has been incorporated to the function 
$\varphi$. 
The parametrized double layer operator, see \eqref{eq:double}, is defined as follows 
\begin{equation}\label{eq:k:b}
(K_k\psi)(t)=\int_0^{2\pi} H_k(t,\tau)\psi(\tau)d\tau:=\int_0^{2\pi}\frac{\partial G_k(\mathbf x(t)-\mathbf x(\tau))}{\partial\mathbf{n}(\mathbf x(\tau))}|{\bf x}'(\tau)| \psi({\tau})d\tau\ 
\end{equation}
Notice that, unlike \eqref{eq:sl:b}, the norm of the parameterization is part of the kernel. 
The parametrized adjoint of the double layer cf. \eqref{eq:adj_double} is given by 
\begin{equation}\label{eq:kt:b}
(K_k^\top\varphi)(t)=\int_0^{2\pi} H^\top_k(t,\tau)\varphi(\tau)d\tau:=\int_0^{2\pi}|{\bf x}'(t)|\frac{\partial G_k(\mathbf x(t)-\mathbf x(\tau))}{\partial\mathbf{n}(\mathbf x(t))} \varphi(t)d\tau. 
\end{equation}
(Observe that $H^\top_k(t,\tau)=H_k(\tau,t)$).  Finally, for the hypersingular operator,  parametrizing the integral in \eqref{eq:normal_double}, multiplying by $|{\bf x}'(t)|$ and adding and subtracting $\frac{1}{4\pi}\ln(4\sin^2((t-\tau)/2)$ and applying integration by parts we {obtain}
\begin{eqnarray} \label{eq:N:b}
(N_k \psi)(t)&:=&\mathrm{PV}\frac{1}{4\pi}\int_0^{2\pi} \cot\frac{t-\tau}{2}\:\psi'(\tau)\,{\rm d}\tau
+\int_0^{2\pi}D_k(t,\tau)\psi(\tau)\,{\rm d}\tau
\end{eqnarray}
with
\begin{eqnarray}
D_k(t,\tau)\!\!&:=&\!\!\!
k^{2} M_k(t,\tau)(\mathbf x'(t))\cdot\ \mathbf x'(\tau)) -
\frac{\partial^2}{\partial t\ \partial\tau}\Big\{M_k(t,\tau)+\frac{1}{4\pi}\ln\Big(\sin^2\frac{t-\tau}{2}\Big)\Big\}.\qquad \label{eq:allkernels} 
\end{eqnarray} 
Note  we have used the fact that
\[                                              
|\mathbf x'(t)||\mathbf x'(\tau)|(\mathbf{n}(\mathbf x(t))\cdot\mathbf{n}(\mathbf x(\tau))=
 \mathbf x'(t))\cdot\ \mathbf x'(\tau).
\]
The integrals operators on $\Gamma$ and their
parametrized versions have been denoted with the same symbols. Furthermore, we will write also
$S_j,K_j, K^\top_j,
N_j^\top$, with $j=1,2$, for $S_{k_j},K_{k_j}, K_{k_j}^\top,
N_{k_j}^\top$. The context will avoid any possible confusion. 

Finally, and for the PSGCSIE equations, we introduce according to \eqref{eq:defPS1}-\eqref{eq:defPS2}
and the strategy  followed in \eqref{eq:sl:b}-\eqref{eq:allkernels}  
these parametrised versions:
\begin{eqnarray}\label{eq:defPS1b}
(PS(N_{\kappa})\psi)(t)&:=&\sum_{n\in\mathbb{Z}}\sigma_0(N_{\kappa})(n)\widehat{\psi}(n) e^{int}, 
\quad (PS(S_{\kappa})\varphi)(t):=\sum_{n\in\mathbb{Z}}\sigma_0(S_{\kappa})(n) \widehat{\varphi}(n) 
e^{int}  
\end{eqnarray} 
{where $\widehat{\psi}(n)$ denote the $n$th Fourier coefficient.

\subsection{Discretization of the GCSIE equations~\eqref{eq:matrix_explicit}\label{dGCSIE}}

 Let us consider $M_j$, $H_j$, $D_j$ the functions  appearing in
the kernels of  $S_j$, $K_j$ and $N_j$, see \eqref{eq:sl:b},
\eqref{eq:k:b} and \eqref{eq:N:b}--\eqref{eq:allkernels}.  
These functions are weakly singular and can 
be written in the form~\cite{KressH} for $j=1,2$
\begin{eqnarray}\label{eq:split}
M_j(t,\tau)&=&M_{j,1}(t,\tau)\ln\left(4\sin^2\frac{t-\tau}{2}\right)+M_{j,2}(t,\tau)\nonumber\\
H_j(t,\tau)&=&H_{j,1}(t,\tau)\ln\left(4\sin^2\frac{t-\tau}{2}\right)+H_{j,2}(t,\tau)\nonumber\\
D_j(t,\tau)&=&D_{j,1}(t,\tau)\ln\left(4\sin^2\frac{
t-\tau } { 2 } \right)+D_{j,2}(t,\tau)
\end{eqnarray}
for bivariate $2\pi$-periodic analytic functions $M_{j,1},H_{j,1},D_{j,1}$ 
and $M_{j,2},H_{j,2},D_{j,2}$. The main idea in the derivation of 
equations~\eqref{eq:split} is to decompose the fundamental solution 
$H_0^{(1)}(z)$ in the form $H_0^{(1)}(z)=J_0(z)+iY_0(z)$ and to use the fact 
that $J_0(z)$ and 
$Y_0(z)-\frac{2}{\pi}J_0(z)\ln z$
are analytic 
functions of $z$; similar decompositions are available for $H_1^{(1)}(z)$. For 
the kernels $H_j^T$ of the operators $K_j^T$ (see \eqref{eq:kt:b}) we use their increased smoothness 
(see Theorem~\ref{regL}) to represent them in the form
\begin{equation}\label{eq:splitHT}
H_j^T(t,\tau)=H_{j,1}^T(t,\tau)\sin^2\left(\frac{t-\tau}{2}\right)\ln\left(4\sin^2\frac{t-\tau}{2}\right)+H_{j,2}^T(t,\tau)
\end{equation}
for $j=1,2$ where the functions $H_{j,1}^T$ and $H_{j,2}^T$ are $2\pi$-periodic analytic functions. In principle, the operators $K_j^T$ can be represented in a similar manner to equations~\eqref{eq:split}. However, we favor the representation of the operators $K_j^T$ given in equations~\eqref{eq:splitHT} in order to handle in a stable manner the composition of operators $K_j^T$ with the hyper-singular operator $N_{\kappa}$ needed for the evaluation of the operators $\tilde{D}_{21}$ in equation~\eqref{eq:matrix_explicit}.

The same splitting strategy, unfortunately, does not work in the case when the kernels involve the Hankel function $H_0^{(1)}(\kappa\:\cdot\ )$ when $\Im\kappa>0$. The reasons are similar to those documented in~\cite{turc1}: the Bessel function $J_0(\kappa|x|)$ grows exponentially as $|x|\to\infty$, while $H_0^1({\kappa|x|})$ actually {\em
decays} exponentially, as $|x|$ increases, and thus the
splitting strategy employed for the kernels $M_j$ generally gives rise to significant cancellation errors if used throughout the
integration domain for the kernels $M_{\kappa}$. In order to avoid subtraction of exponentially large quantities, we
evaluate the operator $S_{\kappa}$ by means of a slight modification of
the approach used for the operators with kernels $M_j,j=1,2$ that contain {\em real} wavenumbers $k_j$: 
we use the truncated decomposition
\begin{eqnarray}
\label{eq:splitImagM}
M_{\kappa}(t,\tau)&=&\chi(|\kappa||{\bf x}(t)-{\bf x}(\tau)|^4)\left\{\tilde{M}_1(t,\tau)\ln\left(4\sin^2\frac{t-\tau}{2}\right)+\tilde{M}_2(t,\tau)\right\}\nonumber\\
&+&(1-\chi(|\kappa||{\bf x}(t)-{\bf x}(\tau)|^4))M_{\kappa}(t,\tau)=M_{\kappa,1}(t,\tau)\ln\left(4\sin^2\frac{t-\tau}{2}\right)+M_{\kappa,2}(t,\tau)\nonumber\\
\end{eqnarray}
where 
 $\chi\in C_0^\infty(\mathbb{R})$ is a function such that $\chi(t)\equiv 1,\ 
|t|\leq 1/2$ and $\chi(t)\equiv 0,\ |t|\geq 1$. It can be checked easily that 
decomposition~\eqref{eq:splitImagM} does not suffer from cancellation errors and 
that the kernels $M_{\kappa,j}(t,\tau),\ j=1,2$ in 
equation~\eqref{eq:splitImagM} are smooth (but not analytic) functions of $t$ 
and $\tau$. Applying the strategy outlined above for the case of the kernel 
$M_{\kappa}$ to the kernels $D_{\kappa}$ and $H_{\kappa}^T$
leads to splittings of the form
\begin{eqnarray}\label{eq:splitK}
D_{\kappa}(t,\tau)&=&
D_{\kappa,1}(t,\tau)\ln\left(4\sin^2\frac{t-\tau}{2}
\right)+D_ {\kappa,2}(t,\tau)\\
H_{\kappa}^T(t,\tau)&=&H_{\kappa,1}^T(t,\tau)\ln\left(4\sin^2\frac{t-\tau}{2}\right)+H_{\kappa,2}^T(t,\tau)
\end{eqnarray}
for $2\pi$-periodic smooth functions 
$D_{\kappa,1},H_{\kappa,1}^T$ and 
$D_{\kappa,2},H_{\kappa,2}^T$. Having described the splitting of 
every boundary integral operator that enter equation~\eqref{eq:matrix_explicit}, 
we express next the parametric form of the operators $\tilde{D}_{ij},i,j=1,2$ 
in equation~\eqref{eq:matrix_explicit}. In order to do this, we need several 
definitions. For a given $2\pi$ periodic function $\psi$ we introduce the 
following operators $A_{\ell}^j,\ell=1,\ldots,8$ and $j=1,2$, 
$A_{\ell}^{\kappa},\ell=1,\ldots,6$, and $T_0$ by
\begin{eqnarray}\label{ops_param1}
(A_{1}^j\psi)(t)&:=&\int_0^{2\pi}M_{j,1}(t,\tau)\ln\left(4\sin^2\frac{t-\tau}{2}\right)\psi(\tau)d\tau\nonumber\\
(A_{2}^j\psi)(t)&:=&\int_0^{2\pi}M_{j,2}(t,\tau)\psi(\tau)d\tau\nonumber\\
(A_{1}^{\kappa}\psi)(t)&:=&\int_0^{2\pi}M_{\kappa,1}(t,\tau)\ln\left(4\sin^2\frac{t-\tau}{2}\right) \psi(\tau)
d\tau\nonumber\\
(A_{2}^{\kappa}\psi)(t)&:=&\int_0^{2\pi}M_{\kappa,2}(t,\tau)\psi(\tau)d\tau\nonumber\\
(A_{3}^j\psi)(t)&:=&\int_0^{2\pi}H_{j,1}(t,\tau)\ln\left(4\sin^2\frac{t-\tau}{2}\right)\psi(\tau)d\tau\nonumber\\
(A_{4}^j\psi)(t)&:=&\int_0^{2\pi}H_{j,2}(t,\tau)\psi(\tau)  d\tau\nonumber\\
(A_{5}^j\psi)(t)&:=&\int_0^{2\pi}H_{j,1}^T(t,\tau)\sin^2\left(\frac{t-\tau}{2}\right)\ln\left(4\sin^2\frac{t-\tau}{2}\right)\psi(\tau)d\tau\nonumber\\
(A_{6}^j\psi)(t)&:=&\int_0^{2\pi}H_{j,2}^T(t,\tau)\psi(\tau)d\tau\nonumber\\
(A_{5}^{\kappa}\psi)(t)&:=&\int_0^{2\pi}H_{\kappa,1}^T(t,\tau)\ln\left(4\sin^2\frac{t-\tau}{2}\right)\psi(\tau)d\tau\nonumber\\
(A_{6}^{\kappa}\psi)(t)&:=&\int_0^{2\pi}H_{\kappa,2}^T(t,\tau)\psi(\tau)d\tau\nonumber\\
(A_{7}^{j}\psi)(t)&:=&\int_0^{2\pi}D_{j,1}(t,\tau)\ln\left(4\sin^2\frac{t-\tau}{2}\right)\psi(\tau)d\tau\nonumber\\
(A_{8}^{j}\psi)(t)&:=&\int_0^{2\pi}D_{j,2}(t,\tau)\psi(\tau)d\tau\nonumber\\
(A_{7}^{\kappa}\psi)(t)&:=&\int_0^{2\pi}D_{\kappa,1}(t,\tau)\ln\left(4\sin^2\frac{t-\tau}{2}\right)\psi(\tau)d\tau\nonumber\\
(A_{8}^{\kappa}\psi)(t)&:=&\int_0^{2\pi}D_{\kappa,2}(t,\tau)\psi(\tau)d\tau\nonumber\\
(T_0\psi)(t)&:=&\frac{1}{4\pi}\int_0^{2\pi}\cot{\frac{\tau-t}{2}}\psi'(\tau)d\tau.
\end{eqnarray}

Using the parametric equations~\eqref{ops_param1} and the identities
\[
\begin{aligned}
 S_j&=A_1^j+A_2^j,\quad  & K_j&=A_3^j+A_4^j,\quad &
K_j^\top &= A_5^j+A_6^j,\quad &  N_j&=T_0+A_7^j+A_8^j\\
 S_{\kappa}&=A_1^{\kappa}+A_2^{\kappa},\quad &
K_{\kappa}^\top &= A_5^{\kappa}+A_6^{\kappa},\quad&  N_{\kappa}&= T_0+A_7^{\kappa}+A_8^{\kappa}
\end{aligned}
\]
we can describe the parametric equations of operators $\tilde{D}_{ij},i,j=1,2$ in equation~\eqref{eq:matrix_explicit}. Define
\begin{eqnarray}\label{eq:DAs}
\tilde{D}_{11}&:=&I-\frac{1}{1+\nu}(A_3^2+A_4^2)+\frac{\nu}{1+\nu}(A_3^1+A_4^1)-\frac{2\nu}{1+\nu}[(A_1^1+A_2^1)(A_7^{\kappa}+A_8^{\kappa}-A_7^1-A_8^1)]\nonumber\\
&&-\frac{2\nu}{1+\nu}[(A_3^1+A_4^1)(A_3^1+A_4^1)]-\frac{2}{1+\nu}[(A_1^2+A_2^2)(A_7^{\kappa}+A_8^{\kappa}-A_7^2-A_8^2)]\nonumber\\
&&-\frac{2}{1+\nu}[(A_3^2+A_4^2)(A_3^2+A_4^2)]\nonumber\\
\tilde{D}_{12} &:=&\frac{1}{1+\nu}(A_1^2+A_2^2)-\frac{1}{1+\nu}
(A_1^1+A_2^1)-\frac{2}{1+\nu}[(A_3^1+A_3^2+A_4^1+A_4^2)(A_1^{\kappa}+A_2^{
\kappa})]\nonumber\\
\tilde{D}_{21}&:=&\frac{\nu}{1+\nu}(A_7^1+A_8^1-A_7^2-A_8^2)-\frac
{2\nu} {1+\nu}[(A_5^1+A_5^2+A_6^1+A_6^2)T_0]\nonumber\\
&&-\frac{2\nu}{1+\nu}[(A_5^1+A_5^2+A_6^1+A_6^2)(A_7^{\kappa}
+A_8^{\kappa})]\nonumber\\
\tilde{D}_{22} &:=&I+\frac{\nu}{1+\nu}(A_5^2+A_6^2)-\frac{1}{1+\nu}
(A_5^1+A_6^1)
-\frac{2}{1+\nu}[(A_7^1+A_8^1-A^{\kappa}_7-A^{\kappa}_8)(A_1^{\kappa}+A_2^
{\kappa})]\nonumber\\
&&-\frac{2\nu}{1+\nu}[(A_7^2+A_8^2-A_7^{\kappa}-A_8^{\kappa})(A_1^{\kappa}
+A_2^{\kappa})
-2[(A_5^{\kappa}+A_6^{\kappa})(A_5^{\kappa}+A_6^{\kappa})].
\end{eqnarray}
Then, the equation we want to approximate is
 \begin{eqnarray}\label{eq:matrix_explicit:b}
\left(\begin{array}{cc}\tilde{D}_{11}&\tilde{D}_{12}\\\tilde{D}_{21}&\tilde{D}_{22}\end{array}\right)\left(\begin{array}{c}a\\b\end{array}\right)&=&\begin{pmatrix}
             f\\
             g
             \end{pmatrix}
\end{eqnarray}
where 
\begin{equation}\label{eq:rhs:b}
\begin{aligned}
  f(t)&:=-(\gamma_D^1 u^{inc})({\bf x}(t)), \qquad&
 g(t)&:=\:-|{\bf x}'(t)|(\gamma_N^1 u^{inc})({\bf x}(t)).\\ 
 a(t)&:=a({\bf x}(t)), &
 b(t)&:=-|{\bf x}'(t)|\:b({\bf x}(t)).\\
\end{aligned} 
\end{equation}
where $(a,b)$ in the right-hand-sides in the bottom line is the solution of \eqref{eq:matrix_explicit}.

We describe next a Nystr\"om method based on trigonometric interpolation that 
follows closely the quadrature method introduced by Kress in~\cite{KressH}, 
which in turn relies on the logarithmic quadrature methods introduced by 
Kussmaul~\cite{kusmaul} and Martensen~\cite{martensen}. We choose 
$n\in\mathbb{N}$ and the equidistant mesh $t_j^{(n)}=\frac{j\pi}{n},\ 
j=0,1,\ldots,2n-1$. With respect to these nodal points the interpolation problem 
in the space $\mathbb{T}_n$ of trigonometric 
polynomials of the form
\[
v(t)=\sum_{m=0}^n a_m\cos{mt}+\sum_{m=1}^{n-1}b_m\sin{mt}
\]
is uniquely solvable~\cite{Kress}. We denote by $P_n:C[0,2\pi]\to 
\mathbb{T}_n$ the corresponding interpolation operator and we will 
use in the error analysis the estimate~\cite[Th. 11.8]{Kress} 
\begin{equation}\label{eq:trig_int}
\|P_n\varphi-\varphi\|_q\leq C_{p,q} n^{q-p}\|\varphi\|_p,\ 0\leq q\leq p,\ \frac{1}{2}<p
\end{equation}  
which is valid for all $\varphi\in H^p[0,2\pi]$ and a constant $C$ depending on $p$ 
and $q$, where  $\|\varphi\|_p$ is the corresponding Sobolev norm of $\varphi$ cf.
\eqref{eq:sobolev}. We use the quadrature rules~\cite{KressH}
\begin{equation}
\int_0^{2\pi}\ln\left(4\sin^2\frac{t-\tau}{2}\right)\psi(\tau)d\tau\approx\int_0^
{2\pi}\ln\left(4\sin^2\frac{t-\tau}{2}\right)(P_n\psi)(\tau)d\tau=\sum_{j=0}^{2n-1}R_j^{(n)}(t)\psi(t_j^{(n)})\label{eq:quad1}
\end{equation}
 where the expressions $R_j^{(n)}(t)$ are given by
$$R_j^{(n)}(t):=-\frac{2\pi}{n}\sum_{m=1}^{n-1}\frac{1}{m}\cos{m(t-t_j^{(n)})}
-\frac{\pi}{n^2}\cos{n(t-t_j^{(n)})}.$$
The quadrature rule in equation~\eqref{eq:quad1} can be easily adapted to the 
case
\begin{eqnarray}\label{eq:quad11}
\int_0^{2\pi}\sin^2\left(\frac{t-\tau}{2}\right)\ln\left(4\sin^2\frac{t-\tau}{2}
\right)\psi(\tau)d\tau&\approx&\int_0^{2\pi}\sin^2\left(\frac{t-\tau}{2}
\right)\ln\left(4\sin^2\frac{t-\tau}{2}\right)(P_n\psi)(\tau)d\tau\nonumber\\
&=&\sum_{j=0}^{2n-1}Q_j^{(n)}(t)\psi(t_j^{(n)})
\end{eqnarray}
 which is relevant to evaluation of operators $A_5^j,j=1,2$. In 
equation~\eqref{eq:quad11} the expressions $Q_j^{(n)}(t)$ are given by
\[
 Q_j^{(n)}(t):=\frac{1}{2n}\left(I(0)+2\sum_{m=1}^{n-1}I(m)\cos{m(t-t_j^{(n)})}
+I(n)\cos{n(t-t_j^{(n)})}\right)
\]
in terms of the coefficients $I(m)$ which are defined for $m\in\mathbb{Z}$ as
\[
I(m):=\frac{1}{2\pi}\int_0^{2\pi}\sin^2\frac{t}{2}\ln\left(
4 \sin^2\frac { t } { 2 } 
\right)e^{imt}dt=\begin{cases}\frac{1}{2} & \text{if $m=0$,}\\ -\frac{3}{8} & 
\text{if $m=\pm 1$},\\ 
\frac{1}{4}\left(\frac{1}{|m+1|}+\frac{1}{|m-1|}-\frac{2}{|m|}\right) & 
\text{otherwise.} \end{cases}
\]
For the evaluation of the operator $T_0$ when applied on 
trigonometric polynomials $\varphi_n\in \mathbb{T}_n$ 
we use cf. \cite{KressH}
\begin{equation}\label{eq:quad3}
  \frac{1}{4\pi}\mathrm{PV}\int_0^{2\pi}\cot{\frac{t-\tau}{2}}\varphi_n'(\tau)d\tau=
  \sum_{j=0}^{2n-1}T_j^{(n)}(t)\varphi_n(t_j^{(n)})
\end{equation} 
%
where
$$T_j^{(n)}(t)=\frac{1}{2n}\sum_{m=1}^{n-1} m\ 
\cos{m(t-t_j^{(n)})}+\frac{1}{4}\cos{n(t-t_j^{(n)})}.$$
Using the quadrature rules~\eqref{eq:quad1} we define the numerical quadrature 
operators for operators whose kernels are a product of a singular logarithmic 
term and an infinitely differentiable function. More precisely, for operators of 
the type
\begin{equation}\label{eq:typeA}
(A\varphi)(t)=\int_0^{2\pi}\ln\left(4\sin^2\frac{t-\tau}{2}\right)K(t,
\tau)\varphi(\tau)d\tau
\end{equation}
where $K(t,\tau)$ is infinitely differentiable in both variables $t$ and $\tau$, 
we define its numerical quadrature operator by
\begin{equation}\label{eq:typeAn}
(A_n\varphi)(t):=\int_0^{2\pi}\ln\left(4\sin^2\frac{t-\tau}{2}\right)(P_nK(t,
\cdot)\varphi)(\tau)d\tau
=\sum_{j=0}^{2n-1}R_j^{(n)}(t)K(t,t_j^{(n)})\varphi(t_j^{(n)}). 
\end{equation}
We use the generic operators $A_n$ introduced above to define the quadrature 
operators $A_{1,n}^j,\ A_{3,n}^j,\ A_{7,n}^j$ for $j=1,2$ as well as quadrature 
operators  $A_{1,n}^{\kappa},\ A_{5,n}^{\kappa},\ A_{7,n}^{\kappa}$. Considering
\eqref{eq:quad11}, for
any operator of the form 
\begin{equation}\label{eq:typeB}
(B\varphi)(t)=\int_0^{2\pi}\sin^2\left(\frac{t-\tau}{2}
\right)\ln\left(4\sin^2\frac{t-\tau}{2}\right)H(t,\tau)\varphi(\tau)d\tau,
\end{equation}
with $H(t,\tau)$ an infinitely 
differentiable $2\pi$ periodic function in both variables $t$ and $\tau$, we will
introduce
the discrete approximation given 
by
\begin{equation}\label{eq:typeBn}
(B_n\varphi)(t):=\int_0^{2\pi}\sin^2\left(\frac{t-\tau}{2}\right)\ln\left(4\sin^2\frac{t-\tau}{2
}\right)(P_nH(t,\cdot)\psi)(\tau)d\tau
=\sum_{j=0}^{2n-1}Q_j^{(n)}(t)H(t,t_j^{(n)})\psi(t_j^{(n)}).
\end{equation}
This strategy is applied to define the quadrature 
operators $A_{5,n}^j,\ j=1,2$. 

The trapezoidal rule 
\[
\int_0^{2\pi}\psi(\tau)d\tau\approx\int_0^{2\pi}(P_n\psi)(\tau)d\tau=\frac{\pi}{n}
\sum_{j=0}^{2n-1}\psi(t_j^{(n)}).
\]
is applied to define 
quadrature operators for operators of the form
\begin{equation}\label{eq:typeE}
(E\varphi)(t)=\int_0^{2\pi}J(t,\tau)\varphi(\tau)d\tau
\end{equation}
for smooth kernels $J(t,\tau)$ as
\begin{equation}\label{eq:typeEn} 
(E_n\varphi)(t):=\int_0^{2\pi}(P_nJ(t,\,\cdot\,)\varphi)(\tau)d\tau=\frac{\pi}{n}\sum_
{j=0}^{2n-1}J(t,t_j^{(n)})\varphi(t_j^{(n)}).
\end{equation}
This approach is followed to define the quadrature operators 
$A_{2,n}^j,\ A_{4,n}^{j},\ A_{6,n}^j,\ A_{8,n}^j$ for $j=1,2$ as well as  $A_{2,n}^{\kappa},\ 
A_{6,n}^{\kappa}$ and $A_{8,n}^{\kappa}$. Having described all the types of integral operators that 
enter the definition of operators $\tilde{D}_{ij},i,j=1,2$ in 
equation~\eqref{eq:DAs}, we present their mapping properties: 
\begin{align*}
&A:H^p[0,2\pi]\to 
H^{p+1}[0,2\pi] ,&  B:&H^p[0,2\pi]\to H^{p+3}[0,2\pi]  ,\\
&E:H^p[0,2\pi]\to 
H^{p+s}[0,2\pi] ,&T_0:&H^p[0,2\pi]\to H^{p-1}[0,2\pi]
\end{align*}
are continuous for all 
$p$ and $s\ge 0$.

 With these notations in hand, we derive the approximating equation to the 
 GCSIE formulation~\eqref{eq:matrix_explicit} we obtain the 
following linear system
\begin{eqnarray}\label{eq:discrete_02}
a_n&-&\frac{1}{1+\nu}P_n(A_{3,n}^2+A_{4,n}^2)a_n+\frac{\nu}{1+\nu}
(A_{3,n}^1+A_{4,n}^1)a_n\nonumber\\
&-&\frac{2\nu}{1+\nu}P_n(A_{1,n}^1+A_{2,n}^1)(A_{
7,n}^{\kappa}+A_{8,n}^{\kappa}-A_{7,n}^1-A_{8,n}^1)a_n\nonumber\\
&-&\frac{2\nu}{1+\nu}P_n(A_{3,n}^1+A_{4,n}^1)(A_{
3,n}^1+A_{4,n}^1)a_n-\frac{2}{1+\nu}P_n(A_{1,n}^2+A_{2,n}^2)(A_{
7,n}^{\kappa}+A_{8,n}^{\kappa}-A_{7,n}^2-A_{8,n}^2)a_n\nonumber\\
&-&\frac{2}{1+\nu}P_n(A_{3,n}^2+A_{4,n}^2)(A_{
3,n}^2+A_{4,n}^2)a_n\nonumber\\
&+&\frac{1}{1+\nu}P_n(A_{1,n}^2+A_{2,n}^2)b_n-\frac{1}{1+\nu}P_n(A_{1,n}^1+
A_{2,n}^1)b_n\nonumber\\
&-&\frac{2}{1+\nu}P_n(A_{3,n}^1+A_{3,n}^2+A_{4,n}^1+A_{4,n}^2)(A_
{ 1 , n } ^ { \kappa }+A_{2,n}^{\kappa})b_n=P_nf\nonumber\\
b_n&+&\frac{\nu}{1+\nu}P_n(A_{5,n}^2+A_{6,n}^2)b_n-\frac{1}{1+\nu}P_n(A_{5,n}^1+A_{6,n}^1)b_n\nonumber\\
&-&\frac{2}{1+\nu}P_n(A_{7,n}^1+A_{8,n}^1-A_{7,n}^{\kappa}-A_{8,n}^{\kappa}
)(A_{1,n}^{\kappa}+A_{2,n}^{\kappa})
b_n\nonumber\\
&-&\frac{2\nu}{1+\nu}P_n(A_{7,n}^2+A_{8,n}^2-A_{7,n}^{\kappa}-A_{8,n}^{
\kappa})(A_{1,n}^{\kappa}+A_{2,n}^{\kappa})
b_n-2P_n(A_{5,n}^{\kappa}+A_{6,n}^{\kappa})
(A_{5,n}^{\kappa}+A_{6,n}^{ \kappa})b_n\nonumber\\
&+&\frac{\nu}{1+\nu}P_n(A_{7,n}^1+A_{8,n}^1-A_{7,n}^2-A_{8,n}
^2)a_n\nonumber\\
&-&\frac{2\nu}{1+\nu}P_n(A_{5,n}^1+A_{5,n}^2+A_{6,n}^1+A_{6,n}^2)
T_ { 0}a_n\nonumber\\
&-&\frac{2\nu}{1+\nu}P_n(A_{5,n}^1+A_{5,n}^2+A_{6,n}^1
+A_{6,n}^2)(A_{7,n}^{\kappa}+A_{8,n}^{\kappa})a_n=P_ng.\nonumber\\
\end{eqnarray}
Observe that necessarily $a_n,b_n$ are trigonometric polynomials. 
Therefore, $T_{0}a_n$  in the equation above can be computed using 
\eqref{eq:quad3}.

We will now prove 
the convergence 
$\|a_n-a\|_p\to 0$ and $\|b_n-b\|_p\to 0,$ as $ n\to\infty$.
To this end, we use several results, which extends results established 
in \cite{Kress,turc1} for convergence of operators \eqref{eq:typeAn}, 
\eqref{eq:typeBn} and \eqref{eq:typeEn} to the continuous operators  
\eqref{eq:typeA}, 
\eqref{eq:typeB} and \eqref{eq:typeE}. 

Let $\alpha\geq 1$ and define the following Banach space of $2\pi$ periodic functions 
\begin{equation}\label{eq:cond_rho}
W^\alpha[0,2\pi]=\{\rho\in L^2[0,2\pi]:\ \triple{\rho}_\alpha<\infty\},\quad \triple{\rho}_\alpha:=\sup_{m\in\mathbb{Z}}\widehat{\rho}(m)(1+|m|^2)^{\frac{\alpha}{2}}
\end{equation}
where $\widehat{\rho}(m)$ is the $m$th Fourier coefficient of
$\rho$ cf. \eqref{eq:sobolev}. 
For $\alpha\geq 1$ and $\rho\in W^\alpha$ consider the boundary 
integral operators
\[
(\Lambda\varphi)(t):=\int_0^{2\pi} K(t,\tau) \rho(t-\tau)\varphi(\tau)\,{\rm d}\tau
\]
with $K$ a smooth  bivariate  2$\pi$-periodic function. We introduce its numerical approximation
\[
(\Lambda_n\varphi)(t):=\int_0^{2\pi} \rho(t-\tau) P_n(K(t,\cdot) \varphi)(\tau)\,{\rm d}\tau. 
\]
Note that the operators considered here  and their corresponding discretizations
fit into this frame by taking $\rho(\tau)= \ln\sin^2(\tau/2)$ in which case $\rho\in W^1[0,2\pi]$, for 
\eqref{eq:typeA}-\eqref{eq:typeAn}, $\rho=\sin^2(\tau/2) \ln\sin^2(\tau/2)$ in which case $\rho\in W^3[0,2\pi]$ in 
\eqref{eq:typeB}-\eqref{eq:typeBn} and $\rho\equiv 1$ in which case $\rho\in W^\alpha[0,2\pi]$ for all $\alpha\geq 1$ in case 
\eqref{eq:typeE}-\eqref{eq:typeEn}.

Next lemma studies the error $\|\Lambda\varphi-\Lambda_n\varphi\|_{p}$.

\begin{lemma}\label{lemma:newLemma}
 Under the assumptions stated above and for all $p\ge \alpha\ge 1$ and   $q\ge 0$  
 there exists $C_{p,q}$ so that
 \[
  \|\Lambda\varphi-\Lambda_n\varphi\|_{p}\le C_{p,q}n^{-q-\alpha}\|\varphi\|_{p+q}.
 \]
\end{lemma}
\begin{proof}
The proof of this result follows closely that of Theorem 12.15 in \cite{Kress}. We start  pointing out several
facts. First, we can expand 
\[
 K(t,\tau)=\sum_{m=-\infty}^\infty k_m(t) e^{i m\tau},\quad 
 k_m(t):=\frac{1}{2\pi}\int_{0}^{ 2\pi } K(t,\tau) e^{-i m \tau}\,{\rm d}\tau.  
\]
Since $K$ is smooth and $2\pi$ periodic, 	
\begin{equation}\label{eq:prop:K}
 \sum_{m\in\mathbb{Z}} (1+|m|)^{P}\|k_{m}\|_P<\infty,\quad \forall P\in\mathbb{R}.
\end{equation}
On the other hand, given that $\rho\in W^\alpha[0,2\pi]$, the mapping 
\[
 \rho*\varphi:=\int_{0}^{2\pi} \rho(\,\cdot\,-\tau)\varphi(\tau)\,d\tau=\sum_{m\in\mathbb{Z}}\widehat{\rho}(m)\varphi_m
 e^{i m\,\cdot\,}
\]
is continuous from $H^p[0,2\pi]$ into $H^{p+\alpha}[0,2\pi]$ with (see \eqref{eq:cond_rho})
\[
 \|\rho*\varphi\|_{p}\le \triple{\rho}_\alpha\|\varphi\|_{p-\alpha}.
\]
The last ingredient in this proof is the bound 
\begin{equation}
\label{eq:prop:avarphi}
 \|a\varphi\|_{q}\le C_q \|a\|_{\max\{|q|,1\}} \|\varphi\|_q.
\end{equation}
(see \cite[Lemma 5.13.1]{Saranen}). Then, 
\[
   \big(\Lambda\varphi-\Lambda_n\varphi)(t)=\sum_{m\in\mathbb{Z}} k_m(t) 
\big(\rho*
   \big[e^{i m\,\cdot\,} \varphi-P_{n}(e^{i m\,\cdot\,} \varphi)]\big)(t), 
\]
and therefore, 
\begin{eqnarray} 
 \|\Lambda\varphi-\Lambda_n\varphi\|_{p}&\le& \sum_{m\in\mathbb{Z}} \|k_m \: \rho*
   \big[e^{i m\,\cdot\,} \varphi-P_{n}(e^{i m\,\cdot\,} \varphi)\big]\|_{p}\le C_p \sum_{m\in\mathbb{Z}} \|k_m\|_{p} 
   \|\rho*
   \big[e^{i m\,\cdot\,} \varphi-P_{n}(e^{i m\,\cdot\,} \varphi)\big]\|_{p}\nonumber \\
   &\le& C_p \triple{\rho}_\alpha \sum_{m\in\mathbb{Z}} \|k_m\|_{p} \:
   \big\| e^{i m\,\cdot\,} \varphi-P_{n}(e^{i m\,\cdot\,} \varphi) \big\|_{p-\alpha}.\label{eq:Lamda:01}
\end{eqnarray}
Inequality \eqref{eq:prop:avarphi} combined with  {the fact that}
$\|e^{i m\,\cdot\,}\|^2_{q}=(1+|m|^{2})^q$ and
estimate
\eqref{eq:trig_int} yield
\[
\|e^{i m\,\cdot\,} \varphi-P_{n}(e^{i m\,\cdot\,} \varphi) 
\|_{p-\alpha}\le C_{p,q} n^{-q-\alpha} 
(1+|m|^2)^{q/2} \|\varphi\|_{p+q}.
\]
(Note that $p-\alpha\ge 0$ and $p+q\ge 1$). Plugging this inequality in 
\eqref{eq:Lamda:01}, we finally obtain
\[
 \|\Lambda\varphi-\Lambda_n\varphi\|_{p}\le  C_{p,q}\bigg[
 \sum_{m\in\mathbb{Z}} \|k_m\|_{p} (1+|m|)^{q}\bigg] n^{-q-\alpha} \|\varphi\|_{p+q}=:
 C_{p,q}' n^{-q-\alpha}\|\varphi\|_{p+q},
\]
where we have used \eqref{eq:prop:K} to bound the series above. The result is 
now proven.
\end{proof}

\begin{lemma}\label{lemm_est_1}
Let $A,$ $B$ and $E$ integral operators as in \eqref{eq:typeA}, \eqref{eq:typeB} 
and \eqref{eq:typeE} respectively, 
and $A_n$, $B_n$ and $E_n$ the corresponding discretizations according to \eqref{eq:typeAn},
\eqref{eq:typeBn} and \eqref{eq:typeEn}. Then, for all $p\ge 1$ and $q\ge0$,  
\begin{eqnarray}
\|P_nA_n\varphi-A\varphi\|_p&\leq& C_{p,q} n^{-q-1}\|\varphi\|_{p+q},\label{eq:01:lemm_est_1}\\
\|P_nB_n\varphi-B\varphi\|_p&\leq& C_{p,q} n^{-q-\min\{p,3\}}\|\varphi\|_{p+q},\label{eq:02:lemm_est_1} \\
\|P_nE_n\varphi-E\varphi\|_p&\leq& C_{p,q} n^{-q}\|\varphi\|_{p},\quad \forall p\ge 1,\ \forall  q>0.\label{eq:03:lemm_est_1}
\end{eqnarray}
where $C_{p,q}$ are independent of $\varphi$ and $n$.
\end{lemma}
\begin{proof}
Note that
\begin{eqnarray*}
\|P_nA_n\varphi-A\varphi\|_p&\le& 
\|P_n(A_n\varphi-A)\varphi\|_p+
\|P_nA\varphi-A\varphi\|_p\\
&\le& C_{p,q} \big[n^{-q-1}\|\varphi\|_{p+q}+n^{-q-1}
\|A\varphi\|_{p+q+1}\big]\\
&\le& C'_{p,q} n^{-q-1}\|\varphi\|_{p+q}
\end{eqnarray*}
where we have applied Lemma \ref{lemma:newLemma} (with $\alpha=1$) for the first term and \eqref{eq:trig_int}. We have then proved \eqref{eq:01:lemm_est_1}.

Similarly, one can prove \eqref{eq:02:lemm_est_1}  for $p\ge 3$. Assume now that $p\in[1,3)$ and notice that
$B_n=B_nP_n$. Then,  making use of \eqref{eq:trig_int}, the continuity
$B:H^0[0,2\pi]\to H^{3}[0,2\pi]\subset H^{p}[0,2\pi]$,
and \eqref{eq:02:lemm_est_1} with 
$p=3$, we can show 
\begin{eqnarray*}
\|P_n B_n\varphi-B\varphi\|_{p}&\le& \|P_n (B_n P_n\varphi-B P_n\varphi)\|_{p}+\|P_n B P_n\varphi-B P_n\varphi\|_{p}+
\|B P_n\varphi-B\varphi\|_{p}\\
&\le&\|B_n P_n\varphi-B P_n\varphi \|_{3}+C_{p,q} n^{-q-3}\|BP_n\varphi\|_{p+q+3} + C_{p} \|P_n\varphi-\varphi\|_{0}\\
&\le&C_{p,q}' n^{-p-q}\big[n^{p-3}\|P_n\varphi\|_{q+3}+
\|\varphi\|_{p+q}\big]. 
\end{eqnarray*}
We apply the well known inverse estimate $\|P_n\varphi\|_{3+q}\le 
(1+n^2)^{(3-p)/2}\|\varphi\|_{p+q}$ (see e.g. Theorem 8.3.1 in~\cite{Saranen}) in the estimate above to get 
\[
\|P_n B_n\varphi-B\varphi\|_{p}\le   
C_{p,q} n^{-p-q}\|\varphi\|_{p+q}, 
\]
for $p\in[1,3]$. This finishes the proof of \eqref{eq:02:lemm_est_1}.

Regarding the last estimate, recall that $E_n$ fits in the hypothesis of 
Lemma \ref{lemma:newLemma}  with $\rho=1$ and, therefore, for any $\alpha\ge 1$. 
Then proceeding as above (with $q=0$) we derive
 \begin{eqnarray*}
\|P_nE_n\varphi-E\varphi\|_p&\le&  \| E_n\varphi-E\varphi\|_p+
\|P_nE\varphi-E\varphi\|_p\\
&\le& C_{p,q} 
\big[n^{-\alpha}\|\varphi\|_{p}+n^{-\alpha}
\|E\varphi\|_{p+\alpha}\big]\le  C'_{p,q} n^{-\alpha}\|\varphi\|_{p}.
\end{eqnarray*}
The result is then proven.
\end{proof}

\begin{lemma}\label{lemm_est_2}
Let $L_1$ and $L_2$ be operators as in \eqref{eq:typeA}, \eqref{eq:typeB} 
or \eqref{eq:typeE} with $L_{1,n}$, $L_{2,n}$  the discretizations  \eqref{eq:typeAn},
\eqref{eq:typeBn} or \eqref{eq:typeEn}.Then for all $1\le p,\ 0\leq q$ and all 
$\varphi\in H^{p+q}[0,2\pi]$ we have the estimate 
\[
\|P_nL_{1,n}L_{2,n}\varphi-L_1L_2\varphi\|_p\leq 
Cn^{-q-1}\|\varphi\|_{p+q}.
\]
\end{lemma}
\begin{proof}
We note that from the previous lemma it holds
\[
\|P_nL_{i,n}\varphi-L_i\varphi\|_p \leq  
C_{p,q} n^{-q-1}\|\varphi\|_{p+q},\quad i=1,2.
\]
Then
\[
\|P_nL_{1,n}L_{2,n}\varphi-L_1L_2\varphi\|_p\leq 
\|P_nL_{1,n}L_{2,n}\varphi-L_1L_{2,n}\varphi\|_p+\|L_1 L_{2,n}\varphi-L_1L_2\varphi\|_p.
\]
We estimate the second term in the right-hand side of the equation above using the result in Lemma~\ref{lemm_est_1}:
\[
\|L_1L_{2,n}\varphi-L_1L_2\varphi\|_p\leq 
C_{p}\|L_{2,n}\varphi-L_2\varphi\|_p\leq C_{p,q}n^{-q-1}\|\varphi\|_{p+q}.
\]
A direct consequence of the result in Lemma~\ref{lemm_est_1} is that
\[
\|P_nL_{2,n}\varphi\|_{p+q}\leq C\|\varphi\|_{p+q}.
\]
From this estimate and one more application of Lemma~\ref{lemm_est_1} we 
obtain
\[
      \|P_nL_{1,n} L_{2,n}\varphi-L_1 L_{2,n}\varphi\|_p\leq 
L_2n^{-q-1}\|P_nL_{2,n}\varphi\|_{p+q}\leq Cn^{-q-1}\|\varphi\|_{p+q}      
\]
from which the result of the Lemma follows.
\end{proof}

Based on the previous result we establish the following
\begin{lemma}\label{lemm_est_4}
For all $p\ge 1,\ q\geq 0$ and all $\varphi\in H^{p+q}[0,2\pi]$, and with $B$ and 
$B_n$ being as in \eqref{eq:typeB} and 
\eqref{eq:typeBn} respectively, 
we have the 
estimate 
$$\|P_nB_{n}T_0\varphi-BT_0\varphi\|_p\leq 
Cn^{-q-\min\{p-1,2\}}\|\varphi\|_{p+q},$$
with $C$ independent of $\varphi$ and $n$.
Moreover, 
\[
 \|P_nE_{n}T_{0}\varphi-ET_0\varphi\|_p\leq Cn^{-q}\|\varphi\|_{p}.
\]
for all $q$. 
\end{lemma}
\begin{proof} For proving the first result, we apply~\eqref{eq:02:lemm_est_1}  of 
Lemma \ref{lemm_est_1} to get
\[
 \|P_nB_{n}T_0\varphi-BT_0\varphi\|_p\leq 
Cn^{-q+1-\min\{p,3\}}\|T_0 \varphi\|_{p+q-1}\leq 
Cn^{-q-\min\{p-1,2\}}\|\varphi\|_{p+q} 
\]
(Note that the fact $T_0:H^{p}[0,2\pi]\to H^{p-1}[0,2\pi]$ is
continuous has  been
applied in the last step.)

The proof of the  second result is similarly, but using   
\eqref{eq:03:lemm_est_1} instead.
\end{proof}

We are ready to state the stability and convergence of the method. First, let us write the discrete equation~\eqref{eq:discrete_02} in the form 
\[
 \mathcal{\tilde{D}}^n\left(\begin{array}{c}a_n\\b_n\end{array}\right)=\left(\begin{array}{cc}\tilde{D}_{11}^n & \tilde{D}_{12}^n\\ \tilde{D}_{21}^n & \tilde{D}_{22}^n\end{array}\right)\left(\begin{array}{c}a_n\\b_n\end{array}\right)=\left(\begin{array}{c}P_nf\\P_ng\end{array}\right)
\]


We notice that Lemma~\ref{lemm_est_1}--Lemma~\ref{lemm_est_4} imply
that there exists a constant $C_{p,q}>0$ such that for all $i,j=1,2$ we have 
\begin{equation}\label{eq:f_conv0}
\|\tilde{D}_{ij}\psi-\tilde{D}_{ij}^n\psi\|_p\leq C_{p,q} n^{-q-\min\{p-1,1\}}\|\psi\|_{p+q}
\end{equation} 
for all $\psi\in H^{p+q}[0,2\pi],\ 1\leq p,\ 0\leq q$, where the operators 
$\tilde{D}_{ij}, i,j=1,2$ were defined in equation~\eqref{eq:matrix_explicit}.

  \begin{theorem}\label{conv} There exists $n_0>0$ such that 
  for all $n\ge n_0$ 
  equation \eqref{eq:discrete_02} has a unique solution. Moreover, for all $p>1$
  and $q\ge 0$ the exists $C_{p,q}$, independent of $n$ and $(a,b)$ so that the unique
  solution $(a_n,b_n)$ of \eqref{eq:discrete_02} satisfies the estimate
   \[
  \max\{\|a_n-a\|_p,\|b_n-b\|_p\}\leq 
  C_{p,q}n^{-q}\max\{\|a\|_{p+q},\|b\|_{p+q}\}
  \]
  where $(a,b)$ is the solution of the parametrized GCSIE  equation~\eqref{eq:matrix_explicit:b}.
\end{theorem}
%
\begin{proof} From \eqref{eq:f_conv0} we have
\[
 \|\mathcal{\tilde{D}}-\mathcal{\tilde{D}}^n\|_{H^{p}[0,2\pi]\times 
H^p[0,2\pi]\to H^{p}[0,2\pi]\times 
H^p[0,2\pi]}\le C_p n^{1-p}. 
\]
Since $\mathcal{\tilde{D}}:H^{p}[0,2\pi]\times 
H^p[0,2\pi]\to H^{p}[0,2\pi]\times H^p[0,2\pi]$ has a bounded inverse, 
it follows from Neumann series 
considerations (see Theorem 10.1 in~\cite{Kress}) that there exists
$n_0$ so that the matrix operators 
$\mathcal{\tilde{D}}^n$ are also invertible $n\ge n_0$ and the inverse 
operators $(\mathcal{\tilde{D}}^n)^{-1}$ are uniformly bounded.
In particular, this implies  
the existence and uniqueness of solution for any $n$ large enough.

Finally, from the identity 
\[
\left(\begin{array}{c}a_n\\b_n\end{array}\right)-\left(\begin{array}{c}
a\\b\end{array}\right)=(\mathcal{\tilde{D}}^n)^{-1}\left(\left(\begin{array}{c}
P_nf-f\\P_ng-g\end{array}\right)+(\mathcal{\tilde{D}}-\mathcal{\tilde{D}}
^n)\left(\begin{array}{c}a\\b\end{array}\right)\right)
\]
and the uniform continuity of $\tilde{\cal D}^n$
we deduce 
\begin{eqnarray*}
 \|a-a_n\|_{p}+\|b-b_n\|_p&\le& C_{p,q} \big[\|f-P_nf\|_p+\|g-P_ng\|_p+
 n^{-q}\big(\|a\|_{p+q}+ \|b\|_{p+q}\big)\big]\\
 &\le&C_{p,q} n^{-q}\big[\|f\|_{p+q}+\|g\|_{p+q}+\|a\|_{p+q}+ \|b\|_{p+q}\big]\\
 &\le&C'_{p,q} n^{-q}\big[\|a\|_{p+q}+ \|b\|_{p+q}\big], 
\end{eqnarray*}
where we have applied \eqref{eq:trig_int} and that 
$ \left(\begin{array}{c}a\\b\end{array}\right)=\mathcal{\tilde{D}}^{-1}
\left(\begin{array}{c}f\\g\end{array}\right)$, and therefore 
the norms of $(f,g)$ can be bounded by those of $(a,b)$.
\end{proof}

\begin{remark}
In the case when the boundary $\Gamma$ and $u^{inc}$ are smooth, we 
have superalgebraic convergence of 
$(a_n,b_n)$ to 
$(a,b)$ in $H^{p}[0,2\pi]$, that is, it is  of the order $\mathcal{O}(n^{-q})$ 
for all $q>0$. Since the kernels $K(\cdot,\cdot)$ and $H(\cdot,\cdot)$ of the integral 
operators of the type $A$, $B$, and $C$ described above that enter the integral 
equation~\eqref{eq:matrix_explicit} are infinitely differentiable but not 
analytic (because of the use of the cutoff function $\chi$), we do not get 
exponential convergence. 
\end{remark} 

\subsection{Discretization of the CFIESK equations~\eqref{eq:system_trans} and  SCFIE equations~\eqref{eq:single}}

An application of the trigonometric interpolation procedure and the smooth and singular quadratures described in Section~\ref{dGCSIE} leads to the approximating equation of the CFIESK~\eqref{eq:system_trans} equations in the form of the  
following linear system
\begin{eqnarray}\label{eq:discrete_CFIESK:b}
\frac{\nu^{-1}+1}{2}u_n&+&P_n(A_{3,n}^2+A_{4,n}^2)u_n-\nu^{-1}P_n(A_{3,n}^1+A_{4,n}^1)u_n\nonumber\\
&+&\nu^{-1}P_n(A_{1,n}^1+A_{2,n}^1)\lambda_n-\nu^{-1}P_n(A_{1,n}^2+A_{2,n}^2)\lambda_n=-\nu^{-1}P_nf\nonumber\\
\frac{\nu^{-1}+1}{2}\lambda_n&+&P_n(A_7^2+A_8^2)u_n-P_n(A_7^1+A_8^1)u_n\nonumber\\
&+&P_n(A_5^1+A_6^1)\lambda_n
-\nu^{-1}P_n(A_5^2+A_6^2)\lambda_n=-P_ng.\nonumber\\
\end{eqnarray}

Again, here $(f,g)$ are as in \eqref{eq:rhs:b}. Then
\begin{equation}\label{eq:discrete_CFIESK:c}
 u_n \approx  u\circ{\bf x}  ,\quad \lambda_n\approx \lambda := \left(\frac{\partial u}{\partial n}\right)|{\bf x}' | .
\end{equation}
Observe that necessarily $u_n,\lambda_n$ are trigonometric polynomials. Using the result in Lemma~\ref{lemm_est_2} and an argument similar to the one used to establish Theorem~\ref{conv} we can prove the stability of the
method, in this case in $H^p[0,2\pi]$ with $p\ge 1$, and the corresponding convergence
estimate.

\begin{theorem}\label{convCFIESK}
For all $n$ sufficiently large   the   
approximating equation~\eqref{eq:discrete_CFIESK:b} has a unique solution $\left(u_n,\lambda_n\right)$. Moreover, for any $p\ge 1$ and 
$q>0$ 
we have the 
following error estimate
\[
\max\left\{\|u_n-u\circ{\bf x}\|_p,\left\|\lambda_n-\lambda\right\|_p\right\}\leq 
C_1 n^{-q}\max\left\{\|u\|_{p+q},\left\|\lambda\right\|_{p+q}\right\}
\]
for some constant $C_1=C_1(p,q)$, where $\left(u\circ{\bf x},\lambda\right)$ is the solution of the 
parametrized CFIESK 
equation~\eqref{eq:system_trans} given in \eqref{eq:discrete_CFIESK:c}.
\end{theorem}

Similarly, an application of the trigonometric interpolation procedure and the smooth and singular quadratures described in Section~\ref{dGCSIE} leads to the approximating equation to the SCFIE formulation~\eqref{eq:single} in the form of the  
following linear system which we solve for the trigonometric polynomial 
\begin{eqnarray}\label{eq:discrete_SCFIE}
-\frac{\nu+1}{2}\psi_n&-&P_n(A_{5,n}^2+A_{6,n}^2)(\nu I_n-2A_{5,n}^2-2A_{6,n}^2)\psi_n\nonumber\\
&-&\nu P_n(A_{5,n}^1+A_{6,n}^1)(I_n+2A_{5,n}^2+2A_{6,n}^2)\psi_n\nonumber\\
&+&2P_n(A_{7,n}^1+A_{8,n}^1-A_{7,n}^2-A_{8,n}^2)(A_{1,n}^2+A_{2,n}^2)\psi_n\nonumber\\
&+&i\eta\ \nu\ P_n(A_{1,n}^1+A_{2,n}^1)(I_n+2A_{5,n}^2+2A_{6,n}^2)\psi_n\nonumber\\
&+&i\eta\ P_n(I_n-2A_{5,n}^1-2A_{6,n}^1)(A_{1,n}^2+A_{2,n}^2)\psi_n=-P_ng+i\eta P_nf\nonumber\\
\end{eqnarray}
where $I_n$ represents the identity operator for trigonometric polynomials with
$(f,g)$ as in \eqref{eq:rhs:b}. Clearly, 
\begin{equation}\label{eq:psin_psi}
 \psi_n\approx  \psi:=|{\bf x}'|\:\varphi\circ{\bf x} 
\end{equation}
where $\varphi$ is the solution of \eqref{eq:single}.

Using the result in Lemma~\ref{lemm_est_2} and an argument similar to the one used to establish Theorem~\ref{conv} we get
\begin{theorem}\label{convSCFIE}
For all $n$ sufficiently large, equation~\eqref{eq:discrete_SCFIE} has a unique solution. Moreover, for all $p\ge 1$ and $q\ge 0$ we have the 
following error estimate
\[
\|\psi_n-\psi\|_p\leq 
C_{p,q}n^{-q}\|\psi\|_{p+q}  
\]for some constant $C$, depending only on $p$ and $q$, where $ \psi $ is the solution of the parametrized SCFIE~\eqref{eq:single} defined in \eqref{eq:psin_psi}.
\end{theorem}

\begin{remark}
In the case when the boundary $\Gamma$ is analytic and $u^{inc}$ is analytic, the convergence of $\left(u_n, \lambda_n\right)$ to  $\left(u\circ{\bf x},\lambda\right)$  in $H^{p}[0,2\pi]$ in Theorem~\ref{convCFIESK} and the convergence of $\psi_n$ to $\psi$ in $H^{p}[0,2\pi]$ in Theorem~\ref{convSCFIE} are of order $\mathcal{O}(e^{-ns})$ for some positive constant $s$. The exponential orders of convergence in Theorem~\ref{convCFIESK} and Theorem~\ref{convSCFIE} are obtained by taking into account the fact that estimates~\eqref{eq:trig_int} can be improved to order $\mathcal{O}(e^{-ns}), s>0$ in the case when $g$ is analytic and $2\pi$ periodic. Thus, all of the error estimates presented above can be improved to order $\mathcal{O}(e^{-ns}), s>0$ given that all the kernels of the integral operators that enter equation~\eqref{eq:discrete_02PS} are themselves analytic.
\end{remark}

\subsection{Discretization of the PSGCSIE equations~\eqref{eq:matrix_PSexplicit}}

The discretization of the PSGCSIE equations~\eqref{eq:matrix_PSexplicit} follows the same lines as the discretization of the GCSIE equations~\eqref{eq:matrix_explicit} described in Section~\ref{dGCSIE}. The main differences between the discretization of the GCSIE equations and PSGSIE consist of (a) the operator compositions $(S_1+\nu^{-1}S_2)PS(N_{\kappa})$ and $(N_1+\nu N_2)PS(S_{\kappa})$ that enter the definitions in equation~\eqref{eq:entriesPSexplicit} and (b) the discretization of the principal symbol operators $PS(S_{\kappa})$ and $PS(N_{\kappa})$. With regards to (a), we aim to highlight certain Calder\'on' type identities for operators compositions $S_j\circ PS(N_{\kappa})$ and $N_j\circ PS(S_{\kappa})$ for $j=1,2$. First, we use a suitable decomposition of the kernels $M_j(t,\tau)$ of the operators $S_j,j=1,2$ in the form
\begin{eqnarray} 
M_j(t,\tau)&=&M_{j,1}(t,\tau)\ln\left(4\sin^2\frac{t-\tau}{2}\right)+M_{j,2}(t,\tau)\nonumber\\
&=&-\frac{1}{4\pi}\ln\left(4\sin^2\frac{t-\tau}{2}\right)+M_{j,1}^1(t,\tau)\sin^2\left(\frac{t-\tau}{2}\right)\ln\left(4\sin^2\frac{t-\tau}{2}\right)+M_{j,2}(t,\tau)\qquad\quad \label{eq:splitM}
\end{eqnarray}
where the kernels $M_{j,1}^1(t,\tau),j=1,2$ are analytic in both variables $t$ and $\tau$ in the case when $\Gamma$ is analytic. For a given $2\pi$ periodic function $\psi$ we define then the operators
\begin{eqnarray}
(A_9^j\psi)(t)&=&-\int_0^{2\pi}M_{j,1}^1(t,\tau)\sin^2\left(\frac{t-\tau}{2}\right)\ln\left(4\sin^2\frac{t-\tau}{2}\right)\psi(\tau)d\tau\nonumber\\
(A_0\psi)(t)&=&\frac{1}{4\pi}\int_0^{2\pi}\ln\left(4\sin^2\frac{t-\tau}{2}\right)\psi(\tau)d\tau.
\end{eqnarray}
It follows from their definition that $A_9^j:H^p[0,2\pi]\to H^{p+3}[0,2\pi]$ for all $p$. Given a $2\pi$ periodic density $b^1$ we can write the following operator composition in the form
\begin{eqnarray}
\frac{2\nu}{1+\nu}(S_1+\nu^{-1}S_2)PS(N_{\kappa}) b^1&=&2 A_0PS(N_{\kappa})b^1+\frac{2\nu}{1+\nu}(A_9^1+\nu^{-1}A_9^2)PS(N_{\kappa})b^1\nonumber\\
&+&\frac{2\nu}{1+\nu}(A_2^1+\nu^{-1}A_2^2)PS(N_{\kappa})b^1.\label{eq:approxAt}
\end{eqnarray}
Given the identity~\cite{Kress}
$$\frac{1}{4\pi}\int_0^{2\pi}\ln\left(4\sin^2\frac{t}{2}\right)e^{int}dt=\begin{cases}0& \text{if $n=0$,}\\ -\frac{1}{2|n|}& \text{otherwise,} \end{cases}$$
we can express the operator $2 A_0PS(N_{\kappa})$ in spectral form as
\begin{equation}
2 (A_0PS(N_{\kappa})\phi)(t)=\sum_{n\in\mathbb{Z},n\neq 0}\frac{\sigma_0(N_{\kappa})(n)}{|n|}\hat{\phi}(n)e^{int}\nonumber\\
\end{equation}
where $\phi\in H^p[0,2\pi]$ and $\phi(t)=\sum_{n\in\mathbb{Z}}\hat{\phi}(n)e^{int}$. It follows easily from the definition~\eqref{eq:defPS1b} (see also \eqref{eq:PS}) that
$$\frac{\sigma_0(N_{\kappa})(n)}{|n|}=-\frac{1}{2}+\mathcal{O}(|n|^{-2}),\ |n|\to\infty$$ 
and thus
\begin{equation}\label{eq:second_compPS}
2 (A_0PS(N_{\kappa})\phi)(t)=-\frac{\phi(t)}{2}+(\tilde{A}_0\phi)(t),
\end{equation}
where the operator $\tilde{A}_0$ has the explicit spectral definition
$$(\tilde{A}_0\phi)(t)=\sum_{n\in\mathbb{Z},n\neq 0}\left(\frac{\sigma_0(N_{\kappa})(n)}{|n|}+\frac{1}{2}\right)\hat{\phi}(n)e^{int}$$
and hence $\tilde{A}_0:H^p[0,2\pi]\to H^{p+2}[0,2\pi]$.
We get thus
\begin{eqnarray}\label{eq:first_compPS}
\frac{2\nu}{1+\nu}(S_1+\nu^{-1}S_2)PS(N_{\kappa}) b^1&=&-\frac{b^1}{2}+\tilde{A}_0b^1+\frac{2\nu}{1+\nu}(A_9^1+\nu^{-1}A_9^2)PS(N_{\kappa})b^1\nonumber\\
&+&\frac{2\nu}{1+\nu}(A_2^1+\nu^{-1}A_2^2)PS(N_{\kappa})b^1.
\end{eqnarray}
Next, we use the representation of the operators $N_j,j=1,2$ in parametric form to write
\begin{eqnarray*}
\frac{2}{1+\nu}(N_1+\nu N_2)PS(S_{\kappa})a^1&=&2T_0[PS(S_{\kappa})a^1]+\frac{2}{1+\nu}(A_7^1+\nu A_7^2)PS(S_{\kappa})a^1\nonumber\\
&+&\frac{2}{1+\nu}(A_8^1+\nu A_8^2)PS(S_{\kappa})a^1.
\end{eqnarray*}
Given that 
\[
(T_0e^{im\cdot})(t)=-\frac{|m|}{2}e^{imt},\ m\in\mathbb{Z},
\]
 we can write the composition $2T_0\ PS(S_{\kappa})$ in spectral form as
\begin{equation}
2 (T_0PS(S_{\kappa})\phi)(t)=-\sum_{n\in\mathbb{Z}} |n|\ \sigma_0(S_{\kappa})(n)\hat{\phi}(n)e^{int}.\nonumber\\
\end{equation}
From \eqref{eq:defPS1b} it follows  that
$$|n|\ \sigma_0(S_{\kappa})(n)=+\frac{1}{2}+\mathcal{O}(|n|^{-2}),\ |n|\to\infty$$ 
and thus
\begin{equation}\label{eq:T0Ps}
2 (T_0PS(S_{\kappa})\phi)(t)=-\frac{\phi(t)}{2}+(\tilde{A}_{00}\phi)(t),
\end{equation}
where the operator $\tilde{A}_{00}$ has the explicit spectral definition
\[
(\tilde{A}_{00}\phi)(t)= \sum_{n\in\mathbb{Z}}\left(-|n|\ \sigma_0(S_{\kappa+i\varepsilon})(n)+\frac{1}{2}\right)\hat{\phi}(n)e^{int}
\]
and thus $\tilde{A}_{00}:H^p[0,2\pi]\to H^{p+2}[0,2\pi]$. Hence, we get
\begin{eqnarray}\label{eq:sec_compPS}
\frac{2}{1+\nu}(N_1+\nu N_2)PS(S_{\kappa})b^1&=&-\frac{b^1}{2}+\tilde{A}_{00}b^1+\frac{2}{1+\nu}(A_7^1+\nu A_7^2)PS(S_{\kappa})b^1\nonumber\\
&+&\frac{2}{1+\nu}(A_8^1+\nu A_8^2)PS(S_{\kappa})b^1.
\end{eqnarray}

Using equations \eqref{eq:approxAt}, \eqref{eq:second_compPS},    \eqref{eq:first_compPS},
\eqref{eq:T0Ps}, and~\eqref{eq:sec_compPS}, we apply the quadrature rules described in Section~\ref{dGCSIE} to derive the following discretization of the PSCGSIE equations
\begin{eqnarray}\label{eq:discrete_02PS}
a_n^1&-&\frac{\nu}{1+\nu}P_n(A_{3,n}^1+A_{4,n}^1)a_n^1-\frac{1}{1+\nu}P_n(A_{3,n}^2+A_{4,n}^2)a_n^1\nonumber\\
&-&\tilde{A}_0a_n^1-\frac{2\nu}{1+\nu}P_n(A_9^1+\nu^{-1}A_9^2)PS(N_{\kappa})a_n^1-\frac{2\nu}{1+\nu}P_n(A_2^1+\nu^{-1}A_2^2)PS(N_{\kappa})a_n^1\nonumber\\
&+&\frac{1}{1+\nu}P_n(A_{1,n}^2+A_{2,n}^2)b_n^1+\frac{1}{1+\nu}P_n(A_{1,n}^1+A_{2,n}^1)b_n^1\nonumber\\
&-&\frac{2}{1+\nu}P_n[(A_{3,n}^1+A_{3,n}^2+A_{4,n}^1+A_{4,n}^2)PS(S_{\kappa})b_n^1=P_nf\nonumber\\
b_n^1&+&\frac{\nu}{1+\nu}P_n(A_{5,n}^2+A_{6,n}^2)b_n-\frac{1}{1+\nu}P_n(A_{5,n}^1+A_{6,n}^1)b_n^1\nonumber\\
&-&\tilde{A}_{00}b_n^1-\frac{2}{1+\nu}P_n(A_7^1+\nu A_7^2)PS(S_{\kappa})b^1-\frac{2}{1+\nu}P_n(A_8^1+\nu A_8^2)PS(S_{\kappa})b^1\nonumber\\
&-&\frac{\nu}{1+\nu}P_n(A_{7,n}^2+A_{8,n}^2-A_{7,n}^1-A_{8,n}^1)a_n^1\nonumber\\
&-&\frac{2\nu}{1+\nu}P_n(A_{5,n}^1+A_{5,n}^2+A_{6,n}^1+A_{6,n}^2)PS(N_{\kappa})a_n^1\nonumber=P_ng.
\end{eqnarray}
Observe that necessarily $a_n^1,b_n^1$ are trigonometric polynomials, 
and therefore $PS(N_{\kappa})$ and $PS(S_{\kappa})$
can be easily computed.

Notice that now it is easy to prove
$$\|\tilde{A}_{0,n}\psi-\tilde{A}_{0}\psi\|_p\leq C n^{-q-2}\|\psi\|_{p+q},\qquad \|\tilde{A}_{00,n}\psi-\tilde{A}_{00}\psi\|_p\leq C n^{-q-2}\|\psi\|_{p+q}.$$
Given that the operators $A_9^j,j=1,2$ are of the type covered in Lemma~\ref{lemma:newLemma}, we obtain the following result along the same lines as Theorem~\ref{conv} 
\begin{theorem}\label{conv1} 
For all $n$ sufficiently large the approximating equation~\eqref{eq:discrete_02PS} has a unique solution $(a_n^1,b_n^1)$. Moreover, for $p\ge 1$ and $q\ge 0$, we have the 
following error estimate
\[
\max\{\|a_n^1-a^1\|_p,\|b_n^1-b^1\|_p\}\leq 
C_3n^{-q}\max\{\|a^1\|_{p+q},\|b^1\|_{p+q}\}
\]
for some constant $C_3=C_3(p,q)$, where $(a^1,b^1)$ is the solution of the 
parametrized PSGCSIE equation~\eqref{eq:matrix_PSexplicit} 
defined in \eqref{eq:rhs:b}. In the case when the boundary $\Gamma$ is analytic and $u^{inc}$ is analytic, the convergence of $(a_n^1,b_n^1)$ to $(a^1,b^1)$ in $H^{p}[0,2\pi]$ is of order $\mathcal{O}(e^{-ns})$ for some positive constant $s$.
\end{theorem}

\section{Numerical results: Nystr\"om discretizations
\label{prec}}

We present in this section a variety of numerical results that
demonstrate the properties of the classical formulations CFIESK~\eqref{eq:system_trans}, SCFIE~\eqref{eq:single}, and the regularized combined field integral equations GCSIE~\eqref{eq:matrix_explicit} and PSGCSIE~\eqref{eq:matrix_PSexplicit} constructed in the previous sections. Solutions of the linear systems arising from the Nystr\"om discretizations of the transmission integral equations described in Section~\ref{singular_int} are obtained by means of the fully complex version of the iterative solver GMRES~\cite{SaadSchultz}. For the case of the regularized GCSIE and PSGCSIE formulations we present choices of the complex wavenumber $\kappa$ in each of the cases considered; our extensive numerical experiments suggest that these values of $\kappa$ leads to nearly optimal numbers of GMRES iterations to reach desired (small) GMRES relative residuals. We also present in each table the values of the GMRES relative residual tolerances used in the numerical experiments.
 
We present scattering experiments concerning the following two smooth geometries: (a) a kite-shaped scatterer whose parametrization is given by ${\bf x}(t)=(\cos{t}+0.65\cos{2t}-0.65,1.5\sin{t})$~\cite{KressH}, and (b) a five petal scatterer whose parametrization is given in polar coordinates by $x_1(t)=r(t)\cos{t}$, $x_2(t)=r(t)\sin{t}$ with $r(t)=1+0.3\cos{5t}$.  We note that each of these geometries has a diameter equals to $2$. For every scattering experiment we consider plane-wave incidence $u^{\rm inc}$ and we present maximum far-field errors, that is we choose sufficiently many directions $\frac{\mathbf{x}}{|\mathbf{x}|}$ and for each direction we compute the far-field amplitude $u^{1}_\infty(\hat{\mathbf{x}})$ defined as
\begin{equation}
\label{eq:far_field}
u^{1}(\mathbf{x})=\frac{e^{ik_1|\mathbf{x}|}}{\sqrt{|\mathbf{x}|}}\left(u^{1}_\infty(\hat{\mathbf{x}})+\mathcal{O}\left(\frac{1}{|\mathbf{x}|}\right)\right),\
|\mathbf{x}|\rightarrow\infty.\\
\end{equation}
The maximum far-field errors were evaluated through comparisons of the
numerical solutions $u_\infty^{1, \rm calc}$ corresponding to either formulation with reference solutions $u_\infty^{1,\rm ref}$ by means of the relation
\begin{equation}
\label{eq:farField_error}
\varepsilon_\infty={\rm max}|u_\infty^{1,\rm
calc}(\hat{\mathbf{x}})-u_\infty^{1,\rm ref}(\hat{\mathbf{x}})|
\end{equation}
The latter solutions $u_\infty^{1,\rm ref}$ were produced using solutions corresponding with refined discretizations based on the formulation SCFIE with GMRES residuals of $10^{-12}$ for all other geometries. Besides far field errors, we display the numbers of iterations required by the GMRES solver to reach relative residuals that are specified in each case. We note that in the cases of high-contrast transmission problems with $k_1>k_2$, we observed that the CFIESK formulation requires two orders of magnitude smaller GMRES tolerance residuals in order to achieve for the same discretizations the same level of accuracy as the other formulations considered. We used in the numerical experiments discretizations ranging from 4 to 10 discretization points per wavelength, for frequencies $\omega$ in the medium to the high-frequency range corresponding to scattering problems of sizes ranging from $2.5$ to $81.6$ wavelengths. The columns ``Unknowns'' in all Tables display the numbers of unknowns used in each case, which equal to the value $4n$ defined in Section~\ref{singular_int} for the CFIESK, GCSIE, and PSGCSIE formulations, and $2n$ for the SCFIE formulation. In all of the scattering experiments we considered plane-wave incident fields of direction $d=(0,-1)$. 

As it can be seen in Table~\ref{results0}, our transmission solvers converge with high-order, as predicted by the error analysis in Section~\ref{singular_int}. In Table~\ref{results1} we present computational times required by a matrix-vector product for each of the four formulations CFIESK, SCFIE, GCSIE, and PSGCSIE. The computational times presented were delivered by a MATLAB implementation of the Nystr\"om
discretization on a MacBookPro machine with $2\times 2.3$ GHz
Quad-core Intel i7 with 16 GB of memory. We present computational
times for the kite geometry, as the computational times required by
the five petal geometry considered in this text are extremely similar to
those for the kite geometry at the same levels of discretization. As
it can be seen from the results in Table~\ref{results1}, the
computational times required by a matrix-vector product for the
CFIESK, SCFIE, and PSGCSIE formulations are quite similar, while the computational times required by a matrix-vector product related to the GCSIE
formulation are on average $1.3$ times more expensive than
those required by the other three formulations. 

In Table~\ref{results22} we present scattering experiments in the case of high-contrast materials so that $k_1<k_2$. As it can be seen, solvers based on the {\em single} formulation SCFIE and the regularized formulations GCSIE and PSGCSIE require fewer GMRES iterations than those based on the formulation CFIESK. In terms of total computational times, the solvers based on the PSGCSIE formulations outperform solvers based on the other three formulations in the high frequency regime.   
\begin{table}
\begin{center}\resizebox{!}{1.4cm}
{
\begin{tabular}{|c|c|c|c|c|c|c|c|c|c|}
\hline
Scatterer & Unknowns & \multicolumn{2}{c|}{${\rm CFIESK}$} &\multicolumn{2}{c|}{${\rm SCFIE}$}&\multicolumn{2}{c|}{${\rm GCSIE}$}& \multicolumn{2}{c|}{${\rm PSGCSIE}$}\\
\cline{3-10}
 & & Iter.& $\epsilon_\infty$ &Iter.&$\epsilon_\infty$&Iter.&$\epsilon_\infty$&Iter.&$\epsilon_\infty$\\
\hline
Kite & 256 & 67 & 3.3 $\times$ $10^{-4}$& 39 & 7.2 $\times$ $10^{-4}$ & 53 & 4.0 $\times$ $10^{-4}$ & 52 & 3.9 $\times$ $10^{-4}$\\
Kite & 512 & 67 & 8.0 $\times$ $10^{-8}$ & 39 & 1.5 $\times$ $10^{-8}$ & 53 & 3.2 $\times$ $10^{-8}$ & 52 & 5.1 $\times$ $10^{-8}$\\
\hline
\hline
Five petal & 256 & 52 & 1.4 $\times$ $10^{-5}$& 32 & 4.1 $\times$ $10^{-5}$ & 45 & 1.2 $\times$ $10^{-5}$ & 42 & 1.2 $\times$ $10^{-5}$\\
Five petal & 512 & 49 & 1.4 $\times$ $10^{-8}$ & 32 & 1.1 $\times$ $10^{-8}$ & 44 & 3.3 $\times$ $10^{-8}$ & 42 & 2.1 $\times$ $10^{-8}$\\
\hline
\end{tabular}
}\caption{\label{results0} High-order accuracy of our solvers for two geometries: kite and five petal geometry. In all the experiments we considered $\nu=\epsilon_1/\epsilon_2$, $\omega=8$, $\epsilon_1=1$, and $\epsilon_2=4$, GMRES residual $10^{-8}$. We note that the number of unknowns used for the SCFIE formulation is half the number of unknowns displayed.}
\end{center}
\end{table}

\begin{table}
\begin{center}
\begin{tabular}{|c|c|c|c|c|c|}
\hline
Geometry & Unknowns & ${\rm CFIESK}$&${\rm SCFIE}$ & ${\rm GCSIE}$ &${\rm PSGCSIE}$\\
\hline
Kite & 512 & 12.99 sec & 12.27 sec& 16.54 sec & 13.97 sec\\ 
Kite & 1024 & 51.55 sec & 50.22 sec & 66.39 sec & 52.39 sec\\ 
\hline
\end{tabular}
\caption{\label{results1} Computational times required by a matrix-vector product for each of the four integral equation formulations of the transmission problems considered in this text.}
\end{center}
\end{table}

\begin{table}
\begin{center}
\resizebox{!}{1.5cm}
{
\begin{tabular}{|c|c|c|c|c|c|c|c|c|c|c|c|}
\hline
$\omega$ & $\epsilon_1$ & $\epsilon_2$ & Unknowns & \multicolumn{2}{c|}{CFIESK} &\multicolumn{2}{c|}{SCFIE}&\multicolumn{2}{c|}{GCSIE}& \multicolumn{2}{c|}{PSGCSIE}\\
\cline{5-12}
 & & & & Iter.& $\epsilon_\infty$ &Iter.&$\epsilon_\infty$&Iter.&$\epsilon_\infty$&Iter.&$\epsilon_\infty$\\
\hline
8 & 1 & 16 & 512 & 79 & 3.5 $\times$ $10^{-4}$& 88 & 7.6 $\times$ $10^{-4}$ & 65 & 5.2 $\times$ $10^{-4}$& 66 & 5.3 $\times$ $10^{-4}$\\
16 & 1 & 16 & 1024 & 122 & 2.3 $\times$ $10^{-3}$ & 121 & 1.1 $\times$ $10^{-3}$ & 93 & 1.6 $\times$ $10^{-3}$ & 91 & 1.8 $\times$ $10^{-3}$ \\
32 & 1 & 16 & 2048 & 176 & 5.6 $\times$ $10^{-4}$ & 152 & 1.5 $\times$ $10^{-3}$ & 112 &  2.2 $\times$ $10^{-3}$ & 109 & 1.9 $\times$ $10^{-3}$\\
64 & 1 & 16 & 4096 & 263 & 7.6 $\times$ $10^{-4}$ & 206 & 1.9 $\times$ $10^{-3}$ & 147 & 1.9 $\times$ $10^{-3}$ & 147 & 2.6 $\times$ $10^{-3}$\\
128 & 1 & 16 & 8192 & 338 & 7.7 $\times$ $10^{-4}$ & 264 & 1.6 $\times$ $10^{-3}$ & 187 & 2.1 $\times$ $10^{-3}$ & 187 & 2.2 $\times$ $10^{-3}$\\
\hline
\end{tabular}
}

\ \\[2ex]

\resizebox{!}{1.5cm}
{
\begin{tabular}{|c|c|c|c|c|c|c|c|c|c|c|c|}
\hline
$\omega$ & $\epsilon_1$ & $\epsilon_2$ & Unknowns & \multicolumn{2}{c|}{CFIESK} &\multicolumn{2}{c|}{SCFIE}&\multicolumn{2}{c|}{GCSIE}& \multicolumn{2}{c|}{PSGCSIE}\\
\cline{5-12}
 & & & & Iter.& $\epsilon_\infty$ &Iter.&$\epsilon_\infty$&Iter.&$\epsilon_\infty$&Iter.&$\epsilon_\infty$\\
\hline
8 & 1 & 16 & 512 & 66 & 2.2 $\times$ $10^{-4}$& 81 & 4.2 $\times$ $10^{-4}$ & 61 & 4.4 $\times$ $10^{-4}$& 63 & 5.8 $\times$ $10^{-4}$\\
16 & 1 & 16 & 1024 & 118 & 1.3 $\times$ $10^{-4}$ & 124 & 3.0 $\times$ $10^{-4}$ & 91 & 2.1 $\times$ $10^{-4}$ & 92 & 2.5 $\times$ $10^{-3}$ \\
32 & 1 & 16 & 2048 & 162 & 7.2 $\times$ $10^{-5}$ & 169 & 3.8 $\times$ $10^{-4}$ & 124 &  3.2 $\times$ $10^{-4}$ & 124 & 3.3 $\times$ $10^{-4}$\\
64 & 1 & 16 & 4096 & 264 & 1.0 $\times$ $10^{-4}$ & 250 & 2.9 $\times$ $10^{-4}$ & 185 & 3.6 $\times$ $10^{-4}$ & 192 & 3.8 $\times$ $10^{-4}$\\
128 & 1 & 16 & 8192 & 348 & 2.0 $\times$ $10^{-4}$ & 350 & 4.1 $\times$ $10^{-4}$ & 241 & 3.3 $\times$ $10^{-4}$ & 247 & 3.4 $\times$ $10^{-4}$\\
\hline
\end{tabular}
}

\caption{\label{results22} Scattering experiments for the kite (top) and five petal (bottom)
  geometry with $\nu=\epsilon_1/\epsilon_2$, and for the CFIESK, SCFIE, GCSIE and PSGCSIE formulations. In the SCFIE formulation we selected $\eta=k_1$. In the regularized formulations GCSIE and PSGCSIE we used $\kappa=(k_1+k_2)/2+i\ \omega $. We note that the number of unknowns used for the SCFIE formulation is half the number of unknowns displayed.}
\end{center}
\end{table}

In Tables \ref{results24} we present scattering experiments in the case of high-contrast materials so that $k_1>k_2$. As it can be seen, solvers based on the {\em single} formulation SCFIE and the regularized formulations GCSIE and PSGCSIE require one order of magnitude fewer GMRES iterations than those based on the formulation CFIESK in order to reach the same level of accuracy. We note that in terms of total computational times, solvers based on the SCFIE, GCSIE, and PSGCSIE also deliver one order of magnitude savings over those based on the CFIESK formulations.

\begin{table}
\begin{center}
\resizebox{!}{1.4cm}
{
\begin{tabular}{|c|c|c|c|c|c|c|c|c|c|c|c|}
\hline
$\omega$ & $\epsilon_1$ & $\epsilon_2$ & Unknowns & \multicolumn{2}{c|}{CFIESK} &\multicolumn{2}{c|}{SCFIE}&\multicolumn{2}{c|}{GCSIE}& \multicolumn{2}{c|}{PSGCSIE}\\
\cline{5-12}
 & & & & Iter.& $\epsilon_\infty$ &Iter.&$\epsilon_\infty$&Iter.&$\epsilon_\infty$&Iter.&$\epsilon_\infty$\\
\hline
8 & 16 & 1 & 512 & 166* & 3.1 $\times$ $10^{-4}$& 34 & 1.5 $\times$ $10^{-4}$ & 30 & 3.8 $\times$ $10^{-4}$& 31 & 4.0 $\times$ $10^{-4}$\\
16 & 16 & 1 & 1024 & 287* & 4.5 $\times$ $10^{-4}$ & 41 & 1.7 $\times$ $10^{-4}$ & 36 & 3.5 $\times$ $10^{-4}$ & 38 & 4.1 $\times$ $10^{-4}$ \\
32 & 16 & 1 & 2048 & 401* & 4.7 $\times$ $10^{-4}$ & 49 & 1.8 $\times$ $10^{-4}$ & 44 & 3.5 $\times$ $10^{-4}$ & 46 & 3.6 $\times$ $10^{-4}$\\
64 & 16 & 1 & 4096 & 668* & 5.3 $\times$ $10^{-4}$ & 58 & 2.1 $\times$ $10^{-4}$ & 52 & 3.8 $\times$ $10^{-4}$ & 54 & 3.6 $\times$ $10^{-4}$\\
128 & 16 & 1 & 8192 & 798* & 1.9 $\times$ $10^{-4}$ & 70 & 2.2 $\times$ $10^{-4}$ & 64 & 3.6 $\times$ $10^{-4}$ & 66 & 3.5 $\times$ $10^{-4}$\\
\hline
\end{tabular}
}

\ \\

\resizebox{!}{1.4cm}
{
\begin{tabular}{|c|c|c|c|c|c|c|c|c|c|c|c|}
\hline
$\omega$ & $\epsilon_1$ & $\epsilon_2$ & Unknowns & \multicolumn{2}{c|}{CFIESK} &\multicolumn{2}{c|}{SCFIE}&\multicolumn{2}{c|}{GCSIE}& \multicolumn{2}{c|}{PSGCSIE}\\
\cline{5-12}
 & & & & Iter.& $\epsilon_\infty$ &Iter.&$\epsilon_\infty$&Iter.&$\epsilon_\infty$&Iter.&$\epsilon_\infty$\\
\hline
8 & 16 & 1 & 512 & 143* & 1.0 $\times$ $10^{-4}$& 25 & 1.2 $\times$ $10^{-4}$ & 25 & 2.3 $\times$ $10^{-4}$& 25 & 2.9 $\times$ $10^{-4}$\\
16 & 16 & 1 & 1024 & 238* & 4.9 $\times$ $10^{-4}$ & 37 & 1.6 $\times$ $10^{-4}$ & 34 & 2.8 $\times$ $10^{-4}$ & 34 & 3.3 $\times$ $10^{-4}$ \\
32 & 16 & 1 & 2048 & 388* & 5.1 $\times$ $10^{-4}$ & 45 & 2.1 $\times$ $10^{-4}$ & 40 & 2.9 $\times$ $10^{-4}$ & 41 & 2.1 $\times$ $10^{-4}$\\
64 & 16 & 1 & 4096 & 630* & 2.7 $\times$ $10^{-4}$ & 54 & 2.0 $\times$ $10^{-4}$ & 48 & 3.0 $\times$ $10^{-4}$ & 50 & 2.6 $\times$ $10^{-4}$\\
128 & 16 & 1 & 8192 & 920* & 3.6 $\times$ $10^{-4}$ & 64 & 2.0 $\times$ $10^{-4}$ & 57 & 3.4 $\times$ $10^{-4}$ & 58 & 3.3 $\times$ $10^{-4}$\\
\hline
\end{tabular}
}
\caption{\label{results24} Scattering experiments for the five petal (top)
and kite (bottom)
  geometry with $\nu=1$, and for the CFIESK, SCFIE, GCSIE, and PSGCSIE formulations. In the SCFIE formulation we selected $\eta=k_1$. In the regularized formulations GCSIE and PSGCSIE we used $\kappa=k_1+i\omega$. The asterisk sign in the CFIESK formulation signifies that the GMRES tolerance residual was set to equal $10^{-6}$ in that case. For all the other formulations we set a GMRES tolerance residual equal $10^{-4}$. We used a lower tolerance residual in the case of the CFIESK formulations in order to achieve the same level of accuracy as the other formulations. One order of magnitude less accurate results were produced when we used a GMRES tolerance residual equal to $10^{-4}$ in that case of the CFIESK formulations. We note that the number of unknowns used for the SCFIE formulation is half the number of unknowns displayed.}
\end{center}
\end{table}

\section*{Acknowledgments}
 Yassine Boubendir gratefully acknowledge support from NSF through contract 
DMS-1319720. Catalin Turc gratefully acknowledge support from NSF through contract DMS-1312169.

\bibliography{biblio2}

\end{document}